\documentclass[10pt]{amsart}

\usepackage{amsmath, amsthm, amssymb, graphicx, enumerate,  
enumitem, datetime, getfiledate, wasysym}
\usepackage{wrapfig}

\newcounter{citedtheorems}

\newcounter{theoremcounter}

\newtheorem{defn}[theoremcounter]{Definition}
\newtheorem{theorem}[theoremcounter]{Theorem}

\newtheorem*{theorem-m}{Theorem \ref{main-theorem}}

\newtheorem*{theorem-x}{Theorem}
\newtheorem{main-claim}[theoremcounter]{Main Claim}
\newtheorem{thm-lit}[citedtheorems]{Theorem}
\newtheorem{defn-lit}[citedtheorems]{Definition}
\newtheorem{fact-lit}[citedtheorems]{Fact}
\newtheorem{fact}[theoremcounter]{Fact}

\newtheorem{cor}[theoremcounter]{Corollary}
\newtheorem{defn-claim}[theoremcounter]{Definition/Claim}

\newtheorem{propn}[theoremcounter]{Proposition}
\newtheorem{concl}[theoremcounter]{Conclusion}
\newtheorem{conv}[theoremcounter]{Convention}
\newtheorem{claim}[theoremcounter]{Claim}

\newtheorem{lemma}[theoremcounter]{Lemma}

\newtheorem{rmk}[theoremcounter]{Remark}

\newtheorem{ntn}[theoremcounter]{Notation}

\newtheorem{example}[theoremcounter]{Example}

\newtheorem{qst}[theoremcounter]{Question}

\newcommand{\trv}{\mathbf{t}}
\newcommand{\mch}{\mathcal{H}}

\newcommand{\eps}{\epsilon}

\newcommand{\trg}{T_{\operatorname{rg}}}
\newcommand{\tdim}{\operatorname{Thr}}
\newcommand{\thr}{\operatorname{Thr}}
\newcommand{\ldim}{\operatorname{Ldim}}
\newcommand{\tlfs}{\triangleleft}
\newcommand{\tlf}{\trianglelefteq}

\newcommand{\br}{\vspace{3mm}}
\newcommand{\sbr}{\vspace{1mm}}

\newcommand{\gee}{\mathcal{G}}

\newcommand{\dom}{\operatorname{dom}}

\newcommand{\lgn}{\operatorname{lg}}

\newcommand{\rstr}{\upharpoonright}
\newcommand{\dual}{\operatorname{dual}}
\newcommand{\vc}{\operatorname{VC}}
\newcommand{\bfs}{\mathbf{s}}

\newcommand{\mbh}{\mathbf{H}}
\newcommand{\mcg}{\mathcal{G}}

\newcommand{\Maj}{\operatorname{Maj}}
\newcommand{\kr}{{\operatorname{k-real}}}
\newcommand{\Cdr}{{\operatorname{Cd-real}}}
\newcommand{\cdr}{{\operatorname{cd-real}}}
\newcommand{\fr}{{\operatorname{(<\omega)-real}}}

\newcommand{\noteom}[1]{{\color{teal}#1}}

\numberwithin{equation}{section}
\numberwithin{theoremcounter}{section}

\title[$\epsilon$-saturation for stable and Littlestone]{$\epsilon$-Saturation for stable graphs \\ and Littlestone classes}

\author{Maryanthe Malliaris, Olga Medrano Mart\'{i}n del Campo, Shay Moran}
\address{Department of Mathematics, University of Chicago} 
\email{mem@math.uchicago.edu}

\thanks{MM: Research partially supported by NSF-BSF 2051825.}

\thanks{OMMC: Research partially supported by Fundaci\'on Alberto y Dolores de Andrade, I.A.P}

\thanks{SM: Shay Moran is a Robert J. Shillman Fellow; he acknowledges support by ISF grant 1225/20, by BSF grant 2018385, by an Azrieli Faculty Fellowship, by Israel PBC-VATAT, by the Technion Center for Machine Learning and Intelligent Systems (MLIS), and by the the European Union (ERC, GENERALIZATION, 101039692). Views and opinions expressed are however those of the author(s) only and do not necessarily reflect those of the European Union or the European Research CouncilExecutive Agency. Neither the European Union nor the granting authority can be held responsible for them.}

\address{Department of Mathematics, University of Chicago}
\email{omedranomdelc@uchicago.edu}

\address{Departments of Mathematics, Computer Science, and Data and Decision Sciences, Technion and Google Research}
\email{smoran@technion.ac.il}

\begin{document}

\begin{abstract}
Any Littlestone class, or stable graph, has finite sets which function as ``virtual elements'':  
these can be seen from the learning side as 
representing hypotheses which are expressible as weighted 
majority opinions of hypotheses in the class, 
and from the model-theoretic side as an approximate finitary version of realizing types. 
We introduce and study the \emph{epsilon-saturation} of a Littlestone class, or stable graph, 
which is essentially the closure of the class under inductively adding all such virtual elements. 
We characterize this closure and prove that under reasonable choices of parameters, 
it remains Littlestone (or stable), though not always of the same Littlestone dimension. This highlights some 
surprising phenomena having to do with regimes of 
epsilon and the relation between Littlestone/stability and VC dimension. 
\end{abstract}

\setcounter{tocdepth}{1}

\maketitle

\section{Introduction}

In infinite model theory, models give rise to types, which describe elements finitely
approximated in the model but possibly not realized there.  In learning theory,
some learning algorithms (called improper learners) necessarily choose hypotheses outside
the given class, which in various examples arise naturally from notions of
dimension or of majority.  Motivated jointly by these ideas 
we develop a quantitative notion of ``closure under aggregation''
which we call $\epsilon$-saturation.  

These definitions make sense for any graph or hypothesis class a priori. 
However, a characteristic of Littlestone, or stable, classes is that every finite 
subset of elements or hypotheses contains a  
``strongly opinionated aggregation'' ($\epsilon$-good or $\epsilon$-excellent set) 
of linear size.  So this is our central case.  
Let $\ell := \ldim(\mch) \geq d := \vc(\mch)$ (defined in \ref{d:littlestone}, \ref{d:vc} below).
We define and characterize the $\epsilon$-saturation $\mch_{\infty}(\epsilon)$
of a Littlestone class $\mch$ for any $\epsilon < 1/2$, and
we prove that $\epsilon$ modulates the complexity of $\mch_\infty(\epsilon)$ in
a strong way.
When $\epsilon < 1/(2^{\ell+2})$,
$\ldim(\mch_{\infty}(\epsilon)) = \ldim(\mch) = \ell$.
For some numerical constant $C < 20$,
if $\epsilon < 1/(Cd)$, $\mch_{\infty}(\epsilon)$
has possibly larger, but finite, Littlestone dimension.
When $ 1/(d+1) < \epsilon < 1/2$, $\mch_{\infty}(\epsilon)$ may have unbounded VC dimension.
We then develop a parallel theory for stable graphs with the more powerful notion of
$\epsilon$-excellent replacing $\epsilon$-good.

The proofs involve theorems and methods from combinatorics, statistical learning, logic, Stable Regularity, and Littlestone Minimax.

\tableofcontents

\section{Notation} 

We will deal with several closely related kinds of objects: graphs, bipartite graphs, 
and hypothesis classes. These need not necessarily be finite, though 
some constructions of the paper ask that they are at most countable. 
To fix notation, 
\begin{itemize}
\item a graph $G = (V,E)$ has a vertex set and an edge set, 
$|G|$ means the size of the vertex set, and 
$E$ is symmetric and irreflexive (i.e., edges are not directed and there are no loops from a 
vertex to itself). 
\item we may write ``$a \in G$'' for $a \in V(G)$ or ``$A \subseteq G$'' for $A \subseteq V(G)$.  
\item for edges and nonedges, we may write ``$a \sim b$'' to abbreviate $(a,b) \in E$ 
and ``$a \not\sim b$'' to abbreviate $(a,b) \notin E$. 
\item a bipartite graph $G = (V_1 \cup V_2, E)$ is given with a partition of the vertex set, and 
here $E \subseteq V_1 \times V_2$ is neither symmetric nor irreflexive. 
\item a hypothesis class can be denoted either $\mch$ or $(X, \mch)$,  
where $X$ is a set and $\mch$ is a set of subsets of $X$. It is also usual to identify 
sets $h \in \mch$ with their characteristic functions, writing $h(a) = 1$ for 
$a \in h$ and $h(a) = 0$ for $a \notin h$. 
\item a hypothesis class $\mch$ can be seen as a bipartite graph in the natural way, letting 
$X$ and $\mch$ be the two sides of the partition of the vertex set and putting an edge 
between $a \in X$ and $h \in \mch$ if and only if $h(a) = 1$. 
\end{itemize}
Notation used locally in sections of the paper will be introduced as needed. 

\section{Preliminaries: good, excellent, stable, Littlestone}\label{s:preliminaries}

In this section we briefly review and motivate good and excellent sets  
 and stable graphs 
following \cite{MiSh:978}, which gives one motivation for our work, and the 
parallel ideas in Littlestone classes.  
For further details, see \cite{MiSh:978},  
\cite{MiSh:E98}, \cite{Terry-Wolf}, \cite{Malliaris-Moran}. \emph{Note: in the full generality of 
the paper below, we will allow weighted good and excellent sets; but for this introductory 
section, the main ideas are communicated without this extra notation.}   
The reader familiar with the area may refer back to this section as needed, or may look ahead to 
\S \ref{s:examples1} for examples relevant to the paper. 

Good sets and excellent sets were introduced by Malliaris and Shelah in \cite{MiSh:978} 
in order to prove the stable regularity lemma, which is a strong version of Szemer\'edi regularity 
which holds for stable graphs (see below).  
In this introduction, good and excellent sets are always finite, although the graph $G$ 
need not be. Informally, 
good sets are almost-atoms: 

\begin{defn}[Good] Given a graph $G$, $\epsilon > 0$, and a finite $B \subseteq G$,  
\[ B \mbox{ is $\epsilon$-good } \]
if for any $a \in G$, either $\{ b \in B: a \sim b \}$ or 
$\{ b \in B:  a \not\sim b \}$ has size $<\epsilon|B|$.  
In the first case,\footnote{Here ``$\trv$'' is for ``truth value,'' expressing the majority opinion.} 
write $\trv(a,B) = 1$, and in the second, $\trv(a,B) = 0$.  
\end{defn}

For any $\epsilon$-good set $B \subseteq G$ and any $a \in G$, $B$ behaves like a 
`virtual' element in the sense that a `virtual' edge or nonedge to $a$ is well defined\footnote{The 
astute reader may wonder about symmetry in the inputs of $\trv$, which we address presently.}  
(according to $\trv(a,B)$). The following definition, also from \cite{MiSh:978}, 
allows for a similar conclusion 
between virtual elements.\footnote{If $A$ and $B$ are $\epsilon$-good, they are already 
in some sense close 
to being a regular pair, at the cost of changing $\epsilon$; as we prefer to keep 
$\epsilon$ fixed, excellence allows for a stronger conclusion.}

\begin{defn}[Excellent]
Given a graph $G$, $\epsilon >0$, $\delta > 0$, and a finite $A \subseteq G$,  
\[ A \mbox{ is $(\epsilon, \delta)$-excellent } \]
if for every $\delta$-good $B \subseteq G$, either 
$\{ a \in A : \trv(a,B) = 0 \}$ or $\{ a \in A : \trv(a,B) = 1 \}$ has size $<\epsilon|A|$.  
In the first case, write $\trv(A, B) = 1$, and in the second, $\trv(A,B) = 0$. 
When $\epsilon = \delta$, write simply  
\[ A \mbox{ is $\epsilon$-excellent.}  \]
\end{defn} 

Excellent sets are good, but the reverse need not be true: see 
\S \ref{example:good-excellent} below. 
The existence of good and excellent sets is characteristic of stable graphs, 
so-called because of stability in model theory, an important concept in 
Shelah's classification theory \cite{Sh:a}. To explain this, first recall the connection 
of various combinatorial definitions:

\begin{defn}[Stable] \label{d:k-stable} A graph $G$ contains a \emph{$k$-half graph} if 
there exist disjoint sets of vertices 
\[ \mbox{ $\{ a_1, \dots, a_k\} $,
$\{ b_1, \dots, b_k \}$ so that $a_i \sim b_j$ if $i<j$ and $a_i \not\sim b_j$ if $j \geq i$} \]
$G$ is called \mbox{$k$-stable}, or sometimes 
\emph{$k$-edge stable} for emphasis, if it contains no $k$-half 
graphs.\footnote{Not all half-graph edges are specified, so $k$-stable forbids a family of configurations.} 
\end{defn} 

\begin{conv}
We say ``$G$ is a stable graph'' to mean that $G$ is $k$-stable for some finite $k$.  
\end{conv}

The next definition is also a finite version of a configuration from 
Shelah's classification theory,  
used to sound out the complexity of a formula [such as the edge relation $E(x,y)$] 
by asking whether the family of instances of the formula [such as the set of neighborhoods] 
allowed for a certain kind of repeated partitioning.  
We say ``special tree'' after \cite{MiSh:978}, others say Shelah tree or Hodges tree.\footnote{The motivation for these names will become evident shortly. Recall that in set theoretic notation, $2 = \{ 0, 1 \}$, 
so ${^{m>}2}$ is the set of binary sequences of length strictly less than $m$, 
${^m 2}$ is the set of binary sequences of length exactly $m$, and 
$\eta^\smallfrown \langle t \rangle \tlf \nu$ means that the string $\eta$ followed by 
$t$ is an initial segment of $\nu$.}   

\begin{defn}[Special trees] \label{d:special-trees}
A graph $G$ contains a \emph{special tree of height $m$}
if there exist disjoint sets of vertices 
$\{ a_\eta : \eta \in {^{m>}2} \}$, $\{ b_\nu : \nu \in {^m 2} \} \subseteq G$,
so that:
\begin{itemize}
\item if $\eta^{\smallfrown} \langle 0 \rangle \tlf \nu$, then $a_\eta \not\sim b_\nu$
\item if $\eta^{\smallfrown} \langle 1 \rangle \tlf \nu$, then $a_\eta \sim b_\nu$
\end{itemize}  
$($and there are no constraints on other edges$)$.
\end{defn}

In Shelah's classification theory, one of the consequences of his celebrated Unstable Formula Theorem (\cite{Sh:a}, Theorem II.2.2, circa 1978) is that special trees and half-graphs are either both bounded or both unbounded.\footnote{That is, in the 
class of models of a complete first-order theory, in 
the bipartite graph given by fixing a first-order formula $\varphi(\bar{x};\bar{y})$ and  
connecting two tuples of the correct length precisely when $\varphi$ holds, 
special trees and half-graphs are either both bounded or both unbounded. For a discussion 
of Shelah's unstable formula theorem and learning theory, see \cite{Malliaris-Moran2}.} 
Around 1981, Hodges looked for a quantitative version, and 
proved (in our language) that there is  
a log bound in each direction. In the next Fact, (1) is essentially sharp but (2) is not known to be. Given 
its role in our argument below, we sketch a proof. 

\begin{fact}[Hodges-Shelah lemma, cf. \cite{hodges}] \label{hs-lemma}
In a graph or bipartite graph $G$: 
\begin{enumerate}
\item if there is a $2^{m+1}$-half graph, there is an $m$-tree 
\item if there is a $2^{k+2}-2$-special tree, there is a $k$-half graph.  
\end{enumerate}
\end{fact}

\begin{proof}[Proof Sketch]
First, from a $k$-half graph $a_1,\dots, a_k, b_1,\dots,b_k$, we find 
a tree of height about $\log k$.  We will let the $b$'s be the leaves of our tree and 
choose $a$'s for the internal nodes. For the root $a_\emptyset$, we choose the midpoint 
of the $a$'s, call it $a_{\lceil k/2 \rceil}$, noting that this element connects to the 
right half of the $b$'s and not to the left half, as befits the root of the tree. 
Continue inductively, choosing the $2^\ell$ internal nodes needed for level $\ell$ 
as the midpoints of the segments with endpoints $\lceil ki/2^\ell \rceil$. 

From an $m$-tree $A = \{ a_\eta : \eta \in {^{m>}2} \}$, 
$B = \{ b_\nu : \nu \in {^m 2} \}$ we can obtain a half-graph of length about $\log m$ 
by induction. There are two preliminary steps.  
First, by induction on $\ell$, we define what it means to be a subtree of height $\ell$. 
A subtree of height $0$ is a leaf, and a subtree of height $\ell+1$ consists of an 
internal node $a_\eta$ which has two distinct extensions $a_{\rho}$ and $a_{\nu}$, 
which are each the root of a subtree of height $\ell$ (here $\eta$ is 
the common initial segment of $\rho$ and $\nu$, but $\rho$ and $\nu$ need not have the same 
length). Second, one proves a Ramsey lemma saying that if we 2-color the internal nodes of 
an $m$-tree then there is a monochromatic subtree of height about $m/2$.  

Now let us choose a first pair of elements for our half-graph. (We will either be choosing 
$a_1, b_1$ or $a_k, b_k$.) Start by choosing any $b \in B$, which induces a 2-coloring of $A$ according to whether or not $b \sim a$. By the Ramsey lemma we can find a large monochromatic subtree. 
Call it $A_*$, and suppose that $b \sim a$ for all $a \in A_*$. In this case our 
$b$ will become $b_k$. Define $a_k$ to be the root of the tree $A_*$. Let $A_{**}$ be the 
subtree whose root is the left successor of $a_k$ in $A_*$, and let $B_{**}$ be its 
set of leaves. Now we have found $a_k, b_k$, and we have a slightly smaller tree 
$A_{**}, B_{**}$ with the property that $b_k \sim a$ for all $a \in A_{**}$, and 
$a_k \not\sim b$ for all $b \in B_{**}$. Note also that had our Ramsey lemma returned 
$\not\sim$, we would have set $b_1 = b$, $a_1$ to be the root of $A_*$,  
$A_{**}$ to be the subtree whose root is the right successor of $a_1$ in $A_*$, 
and $B_{**}$ its leaves . 
We would then have found $a_1, b_1$ and a slightly smaller tree $A_{**}, B_{**}$ with the 
property that $b_1 \not\sim a$ for all $a \in A_{**}$, and $a_1 \sim b$ for all $b \in B_{**}$.    
Either way, we can clearly focus on $A_{**}, B_{**}$ and repeat. 
\end{proof}

\begin{concl} \label{c:stable-orders-trees}
If $G$ is a $k$-stable graph, hence has no $k$-half graphs, 
there is $m=m(k) \leq 2^{k+2}-2$ such that $G$ has no $m$-special trees. 
\end{concl}

This connection of orders and trees was leveraged by Malliaris and Shelah 
in the stable regularity lemma 
to prove existence of good and excellent sets. These proofs 
will be central to our work in this paper so we review them here. 
Since sets of size one are good,  
of course \ref{l:excellent} implies \ref{c:good} and the inclusion of both proofs 
is pedagogical; also, if the $X$ with which we begin 
is too small, say $|X| < 2^m$, we can choose any singleton and finish.

\begin{claim}[Good sets in stable graphs, \cite{MiSh:978}] \label{c:good} 
Let $G$ be a $k$-stable graph, and let $m = m(k)$ be from 
$\ref{c:stable-orders-trees}$. Then for any $0 < \epsilon < \frac{1}{2}$, and any finite 
$X \subseteq G$, there is $Y \subseteq X$, $|Y| \geq \epsilon^{m}|X|$ 
such that $Y$ is $\epsilon$-good.   
\end{claim}

\begin{proof} 
We proceed by induction\footnote{If $X$ is not $\epsilon$-good, 
letting $a_\emptyset \in G$ be such that
$B_{\langle 0 \rangle} := \{ b \in X : b \not\sim a_\emptyset \}$ and
$B_{\langle 1 \rangle} := \{ b \in X : b \sim a_\emptyset \}$ are both nonempty of
size $\geq \epsilon |X|$ starts building a tree.}
 on $\ell < m$. Let $X_\emptyset = X$. 
For each $\eta \in {^\ell 2}$,  
if $X_{\eta}$ is not $\epsilon$-good, let $a_{\eta} \in G$ 
be such that 
$X_{\eta ^\smallfrown \langle 0 \rangle} := \{ x \in X_{\eta} : x \not\sim a_\eta \}$ and  
$X_{\eta ^\smallfrown \langle 1 \rangle} := \{ x \in X_{\eta} : x \sim a_\eta \}$         
are both nonempty of size $\geq \epsilon |X_{\eta}|$. 
Suppose the construction continues through stage $m-1$ so we have defined 
$\{ X_\rho : \rho \in {^{m} 2} \}$ (which are by construction disjoint) 
and $\{ a_\nu : \nu \in {^{m>}2} \}$. 
By hypothesis each $X_\rho \neq \emptyset$, so we may choose 
$x_\rho \in X_\rho$ for each $\rho \in {^{m}2}$. Then 
$\{ a_\nu : \nu \in {^{m>}2} \}$, $\{ x_\rho : \rho \in {^{m}2} \}$ 
form an $m$-tree, contradiction.  So one of the $X_\eta$'s (for 
$\eta$ of length $\leq m-1$) must have been 
$\epsilon$-good, and its size will be $\geq \epsilon^m|X|$.  
\end{proof} 

\begin{lemma}[Excellent sets in stable graphs, \cite{MiSh:978}] \label{l:excellent}
Let $G$ be a $k$-stable graph, and let $m = m(k)$ be from
$\ref{c:stable-orders-trees}$. Then for any $0 < \epsilon < \frac{1}{2^m}$, and any finite
$X \subseteq G$, there is $Y \subseteq X$, $|Y| \geq \epsilon^{m}|X|$
such that $Y$ is $\epsilon$-excellent.
\end{lemma}

\begin{proof} 
We also proceed by induction on $\ell < m$. Let $X_\emptyset = X$.
For each $\eta \in {^\ell 2}$,
if $X_{\eta}$ is not $\epsilon$-excellent, let $A_{\eta} \subseteq G$ be an $\eps$-good set
 such that
$X_{\eta ^\smallfrown \langle 0 \rangle} := \{ x \in X_{\eta} : \trv(x, A_\eta) = 0 \}$ and
$X_{\eta ^\smallfrown \langle 1 \rangle} := \{ x \in X_{\eta} : \trv(x, A_\eta) = 1\}$
are both nonempty of size $\geq \epsilon |X_{\eta}|$.
Suppose the construction continues through stage $m-1$ so we have defined
$\{ X_\rho : \rho \in {^{m} 2} \}$ (which are by construction disjoint) 
and $\{ A_\nu : \nu \in {^{m>}2} \}$.

We now have a `virtual' tree from which we would like to extract an `actual' special tree. 
Since $X_\rho \neq \emptyset$ for every $\rho\in^{m}2$, we first choose any 
$x_\rho \in X_\rho$ for each such $\rho$.  
For each $\nu \in {^{m>}2}$, let $\trv_\rho$ be such that $\trv(x_\rho, A_\nu) = \trv_\rho$. 
This can be considered as a virtual edge, and we want to find an actual element $a_\nu \in A_\nu$ 
which has an actual edge of the correct kind to $x_\rho$, 
for all $\rho$ extending $\nu$. More precisely,  
the set $\{ \rho \in {^m 2} : \nu \tlfs \rho \}$ has size $\leq 2^m$. For any such 
$\rho$, the set $\{ a \in A_\nu : \trv(x_\rho, a) \neq \trv(x_\rho, A_\nu) \}$ of incorrect 
elements has size has size $<\epsilon |A_\nu|$.  
By a union bound and the upper bound on the value of $\epsilon$, after 
each $x_\rho$ has ruled out its small set, some $a_\nu$ remains.  Then
$\{ a_\nu : \nu \in {^{m>}2} \}$, $\{ x_\rho : \rho \in {^{m}2} \}$
form a depth $m$ special tree, contradiction.  So one of the $X_\eta$'s (for $\eta$ of length 
$\leq m-1$)  must have been $\epsilon$-excellent, and its size is at least $\epsilon^m|X|$.
\end{proof}

Note that this result can be extended 
 to any $\epsilon < \frac{1}{2}$ by a theorem of 
Malliaris and Moran \cite{Malliaris-Moran}. 
However, in that result the guarantee on the size of the 
excellent set, while still linear, may be smaller. The proof is beyond the scope of this introduction and involves 
multiplicative-weights algorithms.  

\begin{concl}[\cite{Malliaris-Moran} Claim 2.11, in our language] 
For a graph $G$, the following are equivalent:
\begin{enumerate}
\item $G$ is stable, 
\item for every $\epsilon < 1/2$ there is a constant $c = c(\epsilon) > 0$ such that 
for every finite $A \subseteq G$ there is $B \subseteq A$, $|B| \geq c|A|$, such that 
$|B|$ is $\epsilon$-good. 
\end{enumerate}
\end{concl}

As mentioned above, a hypothesis class $(X, \mch)$ can naturally be seen as a bipartite graph. 
The natural analogue in learning of half-graphs is called threshold dimension 
(as was used in Alon-Livni-Malliaris-Moran \cite{almm}). Note that this dimension is 
a priori one more than stability as $\tdim(G) = k$ means the largest half-graph has size 
$k$, i.e., $G$ is $(k+1)$-stable. 

\begin{defn}[Thresholds]
The threshold dimension of a hypothesis class $\mch$, $\tdim(\mch)$, is defined to be the greatest natural number $\ell$, if it exists, such that there exist $x_1, \dots, x_\ell \in X$ and $h_1, \dots, h_\ell \in \mch$ with $h_i(x_j) = 1$ iff $i<j$. If no such $\ell$ exists\footnote{One can also define ordinal threshold dimension, see REFERENCE. Note that in this definition, $\tdim(\mch) = \infty$ does not distinguish between the case where there is a single infinite half-graph, and arbitrarily large finite half-graphs}, $\tdim(\mch) = \infty$.
\end{defn}

The natural analogue in learning of special trees are Littlestone mistake trees, as 
pointed out by Chase and Freitag \cite{cf}, who used this to give new 
examples of Littlestone classes. Note that here too we record the largest 
$d$ where a tree occurs rather than the smallest $d$ where it does not; also, 
following convention, the root of a mistake tree is at level $1$, not level $0$.  
 
\begin{defn}[Littlestone mistake tree] Let $(X, \mch)$ be a hypothesis class. 
A Littlestone mistake tree of height $d$ is 
a full binary tree whose internal nodes are named by elements of X, 
and which is shattered by $\mch$ in the following sense. 
\begin{enumerate}
\item[(a)] We can represent the process of traversing a root-to-leaf path (i.e., a branch) 
in the tree as a sequence of pairs $\langle (x_i, y_i) : 1 \leq i \leq d \rangle$ 
where $(x_i, y_i)$ records that the name of the current node is $x_i \in X$ and 
from this node we travel right $($if $y_i =1$$)$ or left $($if $y_i = 0$$)$. 
\item[(b)] We say $h \in \mch$ realizes a given branch 
$\langle (x_i, y_i) : 1 \leq i \leq d \rangle$ 
if for all $1 \leq i \leq d$, $h(x_i) = y_i$. 
\item[(c)] We say that $\mch$ shatters the tree if every branch is realized by some $h \in \mch$.  
\end{enumerate} 
\end{defn}

\begin{defn}[Littlestone] \label{d:littlestone}
The Littlestone dimension of a hypothesis class, $\ldim(\mch)$,
is defined to be the greatest natural number
$d$, if it exists, so that there exists a Littlestone mistake tree of height $d$ which is
shattered by $\mch$. If no such $d$ exists, we set $\ldim(\mch) = \infty$. 
$\mch$ is a \emph{Littlestone class} if $\ldim(\mch) < \infty$.   
\end{defn} 

Littlestone classes are well-studied in the learning literature, and 
are characterized by being \emph{online learnable} (see Littlestone \cite{littlestone88} and 
Ben-David, P\'{a}l, Shalev-Schwartz \cite{ben-david09agnostic}). For more on Littlestone 
classes and online learning see \cite{ML-book} \S\S 18.1-18.2.\footnote{Note also that 
$\ldim(\mch)=\infty$ a priori does not distinguish between 
shattered trees of arbitrarily large finite height and shattered trees of infinite height. 
See CITE} 

In what follows, we shall use the analogue of goodness, excellence, 
and the Hodges-Shelah lemma in hypothesis classes in the
natural way. This takes up a line of work beginning with 
Alon-Livni-Malliaris-Moran and Bun-Livni-Moran showing that the 
property of being Littlestone, i.e. stable, gave a qualitative characterization 
of when a hypothesis class was differentially-privately PAC learnable.  
This connection of stability/Littlestone/online learning to statistical 
learning was part of the motivation for the present paper.  

For completeness, recall the Vapnik-Chervonenkis dimension of a class $(X,\mch)$.
\begin{defn}[Shattering]\label{def:shatter}
In a hypothesis class $(X,\mch)$, say a set $S\subseteq X$ is \emph{shattered} if, 
for every $S'\subseteq S$, there exists $h\in\mch$ such that $\{x\in X: h(x)=1\}\cap S = S'$.
\end{defn}

\begin{defn}[VC dimension] \label{d:vc}
The Vapnik-Chervonenkis dimension of a hypothesis class, $\vc(\mch)$, is defined to be the greatest 
natural number $m$, if it exists, so that there exists a set $S$ of size $m$ which is shattered by $\mch$. 
If no such $m$ exists, we set $\vc(\mch)=\infty$. $\mch$ is called a VC class if $\vc(\mch)<\infty$.
\end{defn}\label{def:vcdim}

Note that $\ldim \geq \vc$ for every hypothesis class $\mch$ (and in the graph case, for every graph $\mcg$). This is because shattering a set in the way described above corresponds to a mistake tree with the same structure as a Littlestone mistake tree, with the additional condition that all internal nodes which lie in the same level must be the same element of $X$. 

Finally, it will be useful to define the dual of a hypothesis class. 
Informally, any hypothesis class corresponds to a bipartite graph with $X$ on the left,
$\mch$ on the right, and an edge between $x$ and $h$ if and only if $h(x) = 1$.
The dual class corresponds to the bipartite graph we would get by swapping
the left and right sides without changing any edges.  More formally:

\begin{defn}[Dual hypothesis class] \label{d:dual}
For a hypothesis class $(X, \mch)$, define the dual hypothesis class $\dual(\mch)$ to be
$(Y, \gee)$ where $Y = \mch$ is the domain and $\gee = \{ g_x : x \in X \}$ is
a set of functions from $\mch$ to $\{ 0, 1\}$ given by:
$g_x(h) = 1$ if and only if $h(x) = 1$.
\end{defn}

$\mch$ is a VC class if and only if $\dual(\mch)$ is a VC class, possibly of
larger dimension, and $\mch$ is Littlestone
if and only if $\dual(\mch)$ is Littlestone, possibly of larger dimension. For the first fact, see e.g. Lemma 10.3.4 in \cite{matousek-lectures-discgeo}. The second fact is proved by using a Hodges-Shelah tree to say that
Littlestone dimension is finite iff Threshold dimension is finite:

\begin{concl}[Hodges-Shelah lemma translated to hypothesis classes] \label{hodges-learning} 
Given a hypothesis class $(X, \mch)$:  
\begin{enumerate}
\item If $\tdim(\mch) \geq 2^{m+1}$, then $\ldim(\mch) \geq m$.  
\item If $\ldim(\mch) \geq 2^{k+2}-2$, then $\tdim(\mch) \geq k$. 
\end{enumerate}
\end{concl}

since Threshold dimension is clearly symmetric between a class and its dual.

\br

\section{Definition of $\epsilon$-saturation for $\mch$} \label{s:saturation-H}

In this section we give the central definitions of the paper for hypotheses classes, 
corresponding to saturation and to types. Graphs are considered in \S \ref{s:graphs}.  

These definitions work when $\mch$ is finite, but they don't require it. The 
sets of aggregates we take, however, are always finite.

\begin{conv} 
$(X, \mch)$, which we may abbreviate $\mch$, will usually denote a Littlestone class, with 
\begin{itemize}
\item finite Littlestone dimension $\ell$ 
\item and hence finite VC dimension $d$. 
\end{itemize}
$\mch$ can be finite, but need not be.  
\end{conv}

We will state ``$\mch$ is Littlestone'' when it is used in the main definitions and theorems, and 
will keep the convention of $\ell$ and $d$.  

\begin{conv}
Below, $\epsilon$ and $\delta$ $($possibly with subscripts or other decorations$)$ 
are always in the interval $(0,\frac{1}{2})_{\mathbb{R}}$, subject to further 
restrictions as stated.
\end{conv}

\begin{ntn} \label{n:set-theory}
A notational convention: we think of integers set-theoretically, 
so $n = \{ 0, \dots, n-1 \}$ and ``$i<n$'' means $i$ ranges over the $n$ values $0,\dots, n-1$. 
\end{ntn}

In the present paper, we allow good and excellent sets to be weighted.  By convention 
we ask the weights to be positive, although we could allow zero weights at the cost of modifying 
many instances of ``thus there exists $g$'' to ``thus there exists $g$ of nonzero weight'' 
below.   

\begin{defn} \label{d:eps-hypothesis} Given $\epsilon$ and $\mch$, an \emph{$\epsilon$-hypothesis} 
$($or a \emph{weighted $\epsilon$-good set}$)$ 
of $\mch$ is a finite set of pairs:  
\[ H = \{ (h_i, \gamma_i) : i < k \} \]
where $k \in \mathbb{N}$, and:
\begin{enumerate}
\item $\{ h_i : i < k \} \subseteq \mch$, and  
$\langle h_i : i < k \rangle$ is without repetition. 
\item each $\gamma_i \in (0,1]_{\mathbb{R}}$, and $\sum \{ \gamma_i : i < k \} = 1$.  
\item $($$\epsilon$-goodness$)$ For each $x \in X$, for some $\trv = \trv(x, H) = \trv_H(x) \in \{ 0, 1 \}$, 
\[  \sum \{ \gamma_i : h_i(x) = \trv \} \geq 1-\epsilon \mbox{ and hence } \sum \{ \gamma_i : h_i(x) = 1-\trv \} < \epsilon. \] 
\end{enumerate} 
\end{defn}

Here $\trv$ represents the ``truth value'' of the majority opinion between $x$ and $H$. 

\begin{rmk} As the definition records, any $\epsilon$-hypothesis $H$ determines a function 
$\trv_H : X \rightarrow \{ 0, 1 \}$. Say that $H$ \emph{represents} $\trv_H$.   
A priori, $\trv_H$ need not be in $\mch$. Also, a priori, 
different $\epsilon$-hypotheses may represent the same function. 
\end{rmk}

We will use several different notions of approximation. Since these are well-studied 
we keep the usual names:   

\begin{defn}
Suppose we are given $\mch$ and $f: X \rightarrow \{ 0, 1 \}$. Say $f$ is:  
\begin{enumerate}
\item \emph{$k$-realizable} in $\mch$, for a fixed $k \in \mathbb{N}$, when 
for every $Y \subseteq X$, $|Y| \leq k$, there exists $h \in \mch$ such that 
$h \rstr Y = f \rstr Y$. 

\item \emph{finitely realizable} in $\mch$ if $f$ is $k$-realizable in $\mch$ 
for every finite $k$. 

\item \emph{realized} if $f \in \mch$. 

\end{enumerate}
Write $\mch^\kr$ for the set of functions $f: X \rightarrow \{ 0, 1 \}$ which are $k$-realizable 
in $\mch$, and $\mch^{\fr} = \bigcap \{ \mch^\kr : k < \omega \}$ 
for the set of finitely realizable functions. 
\end{defn}

\begin{rmk}
    Note that in the definition above, $\mch^{\kr}$ could be defined as well if $k$ is a positive real number. Then, $\mch^{\kr}=\mch^{\lfloor k \rfloor\operatorname{-real}}$. This will be used in later sections.
\end{rmk}

The next will play a key role. To see why \ref{d:types-H} is preferable to finite realizability here, 
consider that when $X$ is finite, the only finitely realizable functions are those already 
in $\mch$, but for $\epsilon > 1/|X|$ there may be $\epsilon$-types which are not realized, as in 
\S \ref{sa:matchings} below.  

\begin{defn}[$\epsilon$-types] \label{d:types-H}
Suppose we are given $\mch$, $0 < \epsilon < 1/2$, and $f: X \rightarrow \{ 0, 1 \}$.
We call $f$ an \emph{$\epsilon$-type of $\mch$}, or just \emph{$\epsilon$-type}, if 
$f$ is finitely realized in the class 
\[ \{ \trv_H : H \mbox{ is an $\epsilon$-hypothesis of $\mch$ } \}.  \]
\end{defn}

In a slogan, the function $f$ is an $\epsilon$-type if its restriction to any finite set 
agrees with some $\epsilon$-hypothesis (that is, with some finite weighted $\epsilon$-good set of hypotheses). 

\sbr

The central definition of this paper, in the case of hypothesis classes, is:  

\begin{defn}[The $\epsilon$-saturation of $\mch$] \label{d:epsilon-saturation-H} 
Fix $0 < \epsilon < 1/2$ and $(X,\mch)$ a hypothesis class. Define by induction on $n\leq \omega$ 
a hypothesis class $\mch_n(\epsilon)$ over $X$: 

\begin{enumerate}
\item $\mch_0(\epsilon) = \mch$. 
\item $\mch_{n+1}(\epsilon) = 
\mch_n \cup \{ \trv_H : \mbox{ $H$ is an $\epsilon$-hypothesis of $\mch_n$} \}$. 
\item $\mch_\infty(\epsilon) := \mch_\omega(\epsilon)  = \bigcup \{ \mch_n(\epsilon) : n < \omega \}$. 
\end{enumerate} 
We call $\mch_\infty(\epsilon)$ ``the $\epsilon$-saturation of $\mch$.'' 
\end{defn}

Note that in condition (2) we define $\mch_{n+1}$ as a set of functions, not names 
for functions; so in particular, 
if $H_1$ and $H_2$ are two different $\epsilon$-hypotheses (that is, different finite weighted 
good sets) but $\trv_{H_1} = \trv_{H_2}$ as functions, we only add one copy of this function. 
(In particular, we do not re-add copies of functions already in $\mch$.)  
In later calculations, it will not matter which representation we choose.  

The words ``saturation'' and ``type'' come from model theory. 
In model theoretic language, saturation describes a structure where 
all types of a suitable size are realized.  So this notation suggests a 
connection of $\epsilon$-saturation and $\epsilon$-types, 
which will be justified in due course. 

\begin{rmk}
In Definition $\ref{d:epsilon-saturation-H}$, note $X$ stays fixed. 
If we were to increase both $X$ and $H$ by adding majorities, then in order to retain 
total functions, we would need (in the language of stable regularity) 
to upgrade ``$\epsilon$-good'' in the definition of $\epsilon$-hypothesis above
to ``$\epsilon$-excellent.'' This will be a consideration when we deal with 
graphs rather than hypothesis classes, see \ref{d:epsilon-vertex} below.  
\end{rmk} 

\begin{rmk} 
In Definition $\ref{d:epsilon-saturation-H}$ we could have defined $\mch_\alpha(\epsilon)$ for 
ordinals $\alpha$ by transfinite induction, but this will not be used here. 
\end{rmk}

\begin{rmk}
Observe that as $n \rightarrow \infty$, the set of $\epsilon$-types may change even though $X$ 
stays the same, because the $\epsilon$-hypotheses available in $\mch_n(\epsilon)$ to witness 
finite realizability may be changing. See for instance \S \ref{example15}.  
\end{rmk}

\sbr

\section{Preservation of dimension for small $\epsilon$} \label{s:small-epsilon}

This section proves Theorem \ref{sat:hypothesis} 
which says that Littlestone, threshold, and VC dimension are preserved 
from $\mch$ to $\mch_\infty(\epsilon)$ provided $\epsilon$ is 
sufficiently small relative to the dimension in question, 
by applying ideas from stable regularity (as outlined in \S \ref{s:preliminaries}) 
to weighted good and excellent sets in the natural way.  The theorem will rely on several lemmas.  

For orientation, recall ``$\ldim \geq \ell$''
and ``$\vc \geq d$'' and ``$\thr \geq k$'' are witnessed by the existence of
certain finite configurations: the nodes and leaves of
a shattered tree of height
$\ell$ in the first case, the shattered set of size $d$
(that is, $d$ elements and $2^d$ hypotheses realizing
all possible functions over them) in the second, and
a half-graph of length $k$ in the third.  (We will also use that $VC$ dimension can 
be witnessed by a shattered tree where all the labels at a given level are the same.)  So 
in each case, the Littlestone, VC, or threshold dimension of 
$\mch_n$ is nondecreasing with $n$. Also, since these are finite configurations,  
if the given configuration exists in 
$\mch_\infty$ then already it exists in $\mch_{n+1}$ for some $n$. 

We will say ``special tree'' rather than Littlestone tree to emphasize 
the parallelism, since in some cases we will want to choose the nodes from 
the $X$ side and the leaves from the hypothesis 
side, and in other cases the reverse.   

\begin{rmk}
Recall our set-theoretic notation: for $s \in \mathbb{N}$, ``$i<s$'' means that $i$ takes 
the $s$ values $0, \dots, s-1$.  Recall also from $\ref{d:special-trees}$ that 
${^{s>}2}$ is the set of binary sequences of length strictly less than $s$, and 
${^s 2}$ is the set of binary sequences of length $s$. 
\end{rmk}

For Lemmas \ref{lemma2a}-\ref{lemma3d} let $(X, \mch)$ be an arbitrary but fixed hypothesis class. 
Here we can allow ``$\epsilon \leq \dots$'' instead of ``$\epsilon < \dots$'' just because the 
definition of $\epsilon$-hypothesis has a strict inequality on the smaller side.  

\begin{lemma} \label{lemma2a} For any $s \in \mathbb{N}$ and $\epsilon \leq 1/s$,
if there is a special tree of
height $s$ with the internal nodes chosen from $X$ and the leaves chosen from $\mch_{\infty}$,
then there is a special tree of height $s$ 
with the same internal nodes, and the leaves chosen from $\mch$. 
\end{lemma}

\begin{proof} 
Consider a special tree given by
\[ \langle x_\eta : \eta \in {^{s>}2} \rangle \mbox{ from X and }
\langle y_\rho : \rho \in {^s 2} \rangle \mbox{ from $\mch_{\infty}$. } \] 
Let $n < \omega$ be minimal such that $\{ y_\rho : \rho \in {^s 2} \} \subseteq \mch_{n+1}$.  
For each $y_\rho$ choose an $\epsilon$-hypothesis $H_\rho$ of $\mch_n$ representing the same
function over $X$ as $y_\rho$. So we now have
$\langle x_\eta : \eta \in {^{s>}2} \rangle$
and $\langle H_\rho : \rho \in {^s 2} \rangle$ with $\trv(H_\rho, x_\eta) = 1$
if $\eta^\smallfrown \langle 1 \rangle \tlf \rho$
and $\trv(H_\rho, x_\eta) = 0$ if $\eta^\smallfrown \langle 0 \rangle \tlf \rho$.
For each $\rho \in {^s 2}$, define the error set 
\[ E_\rho = \bigcup \{ (h_i, \gamma_i) \in H_\rho ~:~ \trv(h_i, x_\eta) \neq \trv(H_\rho, x_\eta) \mbox{ for some $\eta$ with }  \eta \tlfs \rho \}. \] 
Observe that $\{ \eta : \eta \tlfs \rho \}$ has size $s$ and that for each $\eta$, the ``error weight'' is $<\epsilon$. 
In other words, $\sum \{ \gamma_i : (h_i,\gamma_i) \in E_\rho \} < \epsilon s \leq (1/s)s = 1$. 
So there is some $(h,\gamma) \in H_\rho \setminus E_\rho$. Set $a_\rho := h$.  
Having thus defined each $a_\rho$, we have found a special tree with 
internal nodes $\langle x_\eta : \eta \in {^{s>}2} \rangle$ from $X$ and
leaves $\langle a_\rho : \rho \in {^s 2} \rangle$ from $\mch_n$. 
Either $n=0$, or we contradict our choice of $n$, so either way we are done.
\end{proof}

\begin{lemma} \label{lemma2c}
For any $s \in \mathbb{N}$ and $\epsilon \leq 1/2^s$,  
if there is a special tree of height $s$ where the internal nodes are  
from $\mch_{\infty}$ and the leaves are from $X$, 
there is also a special tree of height $s$ 
with internal nodes from $\mch$ and the same leaves from $X$.  
\end{lemma}

\begin{proof}
Consider $\langle z_\eta : \eta \in {^{s>}2} \rangle$ 
from $\mch_{\infty}$ and $\langle x_\rho : \rho \in {^s 2} \rangle$ from $X$ which form 
a special tree.  
Let $n < \omega$ be minimal so that $\{ z_\eta : \eta \in {^{s>}2} \} \subseteq \mch_{n+1}$.  
For each $\eta \in {^{s>}2}$ choose an $\epsilon$-hypothesis $H_\eta$ of $\mch_n$ representing the 
same type over $X$ as $z_\eta$. For orientation, this means  
$\trv(H_\eta, x_\rho) = 1$
if $\eta^\smallfrown \langle 1 \rangle \tlf \rho$
and $\trv(H_\eta, x_\rho) = 0$ if $\eta^\smallfrown \langle 0 \rangle \tlf \rho$.
For each of the internal indices $\eta \in {^{s>}2}$ there are $\leq 2^s$ leaf indices 
$\rho \in 2^s$ with $\eta \tlf \rho$. 
For each $\eta \in {^{s>}2}$, define the error set 
\[ E_\eta = \{ (h_i, \gamma_i) \in H_\eta ~:~ \trv(h_i,x_\rho) \neq \trv(H_\eta, x_\rho) \mbox{ for some $\rho$ with } \eta \tlfs \rho \}. \]
Then $\sum \{ \gamma_i : (h_i, \gamma_i) \in E_\eta \} < \epsilon 2^s \leq 1$ so there is 
some $(h,\gamma) \in H_\eta \setminus E_\eta$. Set $c_\eta = h$. Having chosen all such $c_\eta$s, 
$\langle c_\eta : \eta \in {^{s>}2} \rangle$ from $H_n$ and 
 $\langle x_\rho : \rho \in {^s 2} \rangle$ from $X$ give a special tree of height $s$. 
\end{proof}

\begin{lemma} \label{lemma3b}
In the notation of the previous proof: suppose that our original tree  
$\langle z_\eta : \eta \in {^{s>}2} \rangle$, $\langle x_\rho : \rho \in {^s 2} \rangle$ 
had the property that for all $\eta_1, \eta_2 \in {^{s>}2}$, 
if $\lgn(\eta_1) = \lgn(\eta_2)$ then $z_{\eta_1} = z_{\eta_2}$. 
Then we can ensure that for all $\eta_1, \eta_2 \in {^{s>}2}$, if $\lgn(\eta_1) = \lgn(\eta_2)$ then  
$c_{\eta_1} = c_{\eta_2}$.   
\end{lemma}

\begin{proof} 
The property says that for each $t < s$, for every $\eta$ of length $t$ we chose the same value. 
For each such $t$, then, make sure to choose the same $H_\eta$ representing each $z_\eta$ of length $t$, 
and denote it by $H_t$. 
Then we simply let the error set for $H_t$ 
collect the wrong votes from all of the $x_\rho$'s, of which there are $2^s$-many. 
\end{proof}

\begin{lemma} \label{lemma3d}
Suppose there is a half-graph consisting of $\langle x_i : i < s \rangle$ from $X$  
and $\langle y_i : i < s \rangle$ from $\mch_\infty$. If $\epsilon \leq 1/s$, 
there is a half-graph consisting of the same $\langle x_i : i < s \rangle$ from $X$ 
and $\langle z_i : i < s \rangle$ from $\mch$.  
\end{lemma}

\begin{proof}
Again, the error set picks up an error of weight strictly less than $\epsilon$ from 
each of the at most $s$ elements on the opposite side. 
\end{proof}

Summarizing: 

\begin{theorem}[Preservation theorems for small $\epsilon$] \label{sat:hypothesis} 
Let $\mch$ be any hypothesis class and $s \in \mathbb{N}$. 

\begin{enumerate}
\item If $\epsilon \leq 1/s$, 
\[ \ldim(\mch) \geq s ~ \iff ~ \ldim(\mch_\infty (\epsilon)) \geq s \]
and also 
\[ \vc(\mch) \geq s ~ \iff ~ \vc(\mch_\infty(\epsilon)) \geq s. \]

\item If $\epsilon \leq 1/2^s$,  
\[ \ldim (\dual(\mch) ) \geq s ~ \iff ~ \ldim( \dual(\mch_\infty(\epsilon))) \geq s \]  
and also 
\[ \vc (\dual(\mch)) \geq s ~ \iff ~ \vc( \dual(\mch_\infty(\epsilon)) ) \geq s. \]

\sbr
\item Hence if $\mch$ is a Littlestone class and $\epsilon \leq 1/(\ldim(\mch)+1)$ then  
\[ \ldim(\mch) = \ldim(\mch_\infty(\epsilon)). \]
\item Hence if $\mch$ is a VC class and $\epsilon \leq 1/(\vc(\mch)+1)$ then 
\[ \vc(\mch) = \vc(\mch_\infty(\epsilon)). \]
\item Note: regardless of $\epsilon$, if $\mch$ is not Littlestone, $\ldim(\mch) = \ldim(\mch_\infty(\epsilon)) = \infty$, and if $\mch$ is not a VC class, 
$\vc(\mch) = \vc(\mch_\infty(\epsilon)) = \infty$.  
\sbr
\item In particular, if $\ldim(\mch) = \ell \in \mathbb{N}$ 
and $\epsilon < 1/2^{\ell+2}$ then $\ldim(\mch) = \ldim(\mch_\infty(\epsilon))$, $\thr(\mch) = \thr(\mch_\infty(\epsilon))$, $\ldim(\dual(\mch)) = \ldim(\dual(\mch_\infty(\epsilon)))$, 
$\vc(\mch) = \vc(\mch_\infty(\epsilon))$ 
and $\vc(\dual(\mch)) = \ldim(\vc(\mch_\infty(\epsilon)))$. 
\end{enumerate}
\end{theorem}

\begin{proof}
\begin{enumerate}
\item[(1)] The first equation is Lemma \ref{lemma2a}. The second is also Lemma \ref{lemma2a}, 
with the extra observation that a hypothesis class $(X, \gee)$ shatters a subset 
$\{ x_0, \dots, x_{s-1} \}$ from $X$ if and only if we can find a special tree of height $s$ 
with internal nodes from $X$ and leaves from $\gee$ where, for each $t<s$, every node at 
level $t$ is equal to $x_t$.  

\item[(2)] As in (1) using Lemmas \ref{lemma2c} and \ref{lemma3b}. 

\item[(3)] By monotonicity, $\ldim(\mch) \leq \ldim(\mch_\infty(\epsilon))$. 
Let $\ell = \ldim(\mch)$. 
By the choice of $\epsilon$ and Lemma $\ref{lemma2a}$, if $\ldim(\mch_\infty(\epsilon)) \geq \ell+1$
then also $\ldim(\mch) \geq \ell+1$, contradiction. 

\item[(4)] As in (3) with $\vc$ in place of $\ldim$.  

\item[(5)] By monotonicity. 

\item[(6)] We know that $\ldim(\mch) \geq \vc(\mch)$ so the results on 
$\ldim$ and $\vc$ of $\mch_\infty(\epsilon)$ and its dual follow from the earlier statements. 
The minor new point is that a small $\epsilon$ in terms of 
$\ldim(\mch)$ helps with the threshold dimension also. 
Let $k = \thr(\mch)$.  Taking the contrapositive of 
Conclusion \ref{hodges-learning}(1) in the case $m=\ell+1$ 
we see that $k < 2^{\ell+2}$ and so the contrapositive of $\ref{lemma3d}$ applies.   
\end{enumerate}
This completes the proof. 
\end{proof}

Note: these results could have been stated in the language of $k$-realizability, which has 
other advantages, but then they would 
not as obviously extend to graphs. 

Before continuing, the reader may wish to look at the section of motivating examples, \S \ref{s:examples1} below. 

\sbr

\section{Characterization of $\epsilon$-saturation for Littlestone classes}

\begin{defn}
Call a hypothesis class $\mbh$ $\epsilon$-saturated if every $\epsilon$-type of $\mbh$ is realized in $\mbh$. 
\end{defn}

The aim of this section is to prove:

\begin{theorem}[$\mch_\infty(\epsilon)$ is $\epsilon$-saturated] \label{t:saturation1} 
Suppose $\mch$ and $\epsilon < 1/2$ are given. Suppose $\mch_\infty(\epsilon)$ is a Littlestone 
class $($hence also $\mch$ is a Littlestone class$)$. Then every $\epsilon$-type of $\mch_\infty(\epsilon)$ is realized in $\mch_\infty(\epsilon)$, i.e. $\mch_\infty(\epsilon)$ is $\epsilon$-saturated.
\end{theorem}

We will invoke the Littlestone Minimax Theorem, 
recently proved by Hanneke, Livni and Moran, which we state here  
in the language of the present paper.\footnote{See \cite{hlm-littlestoneminmax} Corollary 7, where 
finite support is explicitly mentioned; also relevant for us are Lemma 11 and Lemma 12 
of that paper.}  

\begin{thm-lit}[Littlestone Minimax Theorem; Hanneke, Livni, Moran \cite{hlm-littlestoneminmax}] \label{t:lmm}
Suppose $\mbh$ is a Littlestone class and $\epsilon_L < 1/2$. Let 
$g: X \rightarrow \{ 0, 1 \}$ be any function. 
\\ If $(A)$ then $(B)$. 
\begin{enumerate}
\item[(A)] 
For every set $\{ (x_t, w_t) : t < s \}$ where $\{ x_t : t < s \} \subseteq X$, 
$0 < w_t \leq 1$ for each $t<s$, and $\sum_t w_t = 1$,  
there is $h \in \mbh$ so that  
\[ \sum \{ w_t : h(x_t) \neq g(x_t), t < s \} < \epsilon_L. \]
Informally, \emph{``for every finitely supported probability measure on 
$X$ there is some $h \in \mbh$ whose error against $g$ is $<\epsilon_L$.''}  
\item[(B)] 
For every $\epsilon_M$ with $\epsilon_L < \epsilon_M < 1/2$, there exists a finite 
weighted subset $\{ (h_i, \gamma_i) : i < m \}$ of $\mbh$ 
such that: for every $x \in X$, 
\[ \sum \{ \gamma_i : h_i(x) = g(x) \} \geq 1-\epsilon_M. \]
In our language, \emph{``there is an $\epsilon_M$-hypothesis $H$ of $\mbh$ 
representing $g$.''}
\end{enumerate}
\end{thm-lit}\label{thm:littlestone-minimax}

We work towards Theorem \ref{t:saturation1}. 

\begin{claim} \label{claim1a} 
Assume $\mch_\infty(\epsilon)$ is Littlestone. For any $n<\omega$,  
if $g: X \rightarrow \{ 0, 1 \}$ is finitely realized in $\mch_n(\epsilon)$, then 
$g \in \mch_{n+1}(\epsilon)$.
\end{claim}

\begin{proof} 
In this case we can satisfy condition (A) of Littlestone Minimax for 
any $\epsilon_L > 0$, simply by choosing 
$h \in\mch_{n}(\eps)$ to be the function which agrees with $g$ on $\{ x_i : i < s \}$ for some $s<\omega$.  In particular, we can 
take $\epsilon_M = \epsilon$. Let $H$ be the $\epsilon$-hypothesis given by condition (B) where the $h_i$s belong to $\mch_n(\epsilon)$.  
Then $\trv_H$ belongs to $\mch_{n+1}(\epsilon)$, or in other words, 
$g \in \mch_{n+1}(\epsilon)$.  
\end{proof}

\begin{claim} \label{claim2a} 
Assume $\mch_\infty(\epsilon)$ is Littlestone. For any $n<\omega$, 
if $\trv : X \rightarrow \{ 0, 1 \}$ is an $\epsilon$-type of $\mch_n(\epsilon)$, 
then $\trv$ is realized in $\mch_{n+2}(\epsilon)$. 
\end{claim}

\begin{proof}
By construction, $\trv$ will be finitely realized in $\mch_{n+1}(\epsilon)$, hence 
realized in $\mch_{n+2}(\epsilon)$ by Claim \ref{claim1a}.  
\end{proof}

\begin{claim} \label{claim3a}
If $H$ is an $\epsilon$-hypothesis of $\mch_\infty(\epsilon)$, then $\trv_H \in \mch_\infty(\epsilon)$. 
\end{claim}

\begin{proof}
Suppose $H = \{ (h_i, \gamma_i) : i < n \}$. Since it only involves finitely many $h_i$'s, let
$m < \omega$ be such that $\{ h_i : i < n \} \subseteq \mch_m(\epsilon)$.  Clearly $H$ remains an 
$\epsilon$-hypothesis of $\mch_m$, hence $\trv_H \in \mch_{m+1}(\epsilon) \subseteq \mch_\infty(\epsilon)$. 
\end{proof} 

\begin{claim} \label{claim4a}
Assume $\mch_\infty(\epsilon)$ is Littlestone. If $\trv$ is an $\epsilon$-type of $\mch_\infty(\epsilon)$, 
then $\trv$ is finitely realized in $\mch_\infty(\epsilon)$, hence realized in $\mch_\infty(\epsilon)$. 
\end{claim}

\begin{proof} 
By definition of $\epsilon$-type, $\trv$ is finitely realized by $\epsilon$-hypotheses in $\mch_\infty(\epsilon)$.  
So by Claim \ref{claim3a} it is finitely realized in $\mch_\infty(\epsilon)$. In general, let $g: \mch_\infty(\epsilon) \rightarrow \{ 0, 1 \}$ 
be any function which is finitely realized in $\mch_\infty(\epsilon)$. 
Then Littlestone Minimax (A) is always satisfied (for any $0 < \epsilon_L < \epsilon_M$, in particular for $\epsilon_M = \epsilon$, 
since in condition (A) we can  
find $h$ which agrees exactly with $g$ on the given finite domain, hence has zero error whatever the weights). 
So there is an $\epsilon$-hypothesis $H$ of 
$\mch_\infty(\epsilon)$ so that $\trv_H$ represents $g$. By Claim \ref{claim3a}, $\trv_H \in \mch_\infty(\epsilon)$.  
In particular, $\trv$ is realized.  
\end{proof}

\begin{proof}[Proof of Theorem \ref{t:saturation1}] 
Follows from Claim \ref{claim4a}.
\end{proof}

In the next lemma, the intent is that 
all the classes are over the same set $X$.
\begin{lemma} \label{claim5a}
Suppose $\mch_\infty(\epsilon)$ is Littlestone. 
Suppose $\mch \subseteq \mbh$. If $\mbh$ is $\epsilon$-saturated then $\mch_\infty(\epsilon) \subseteq \mbh$.  
Hence $\mch_\infty(\epsilon)$ 
is the smallest $($under inclusion$)$ 
$\epsilon$-saturated hypothesis class containing $\mch$.\footnote{There can of course be  
larger $\epsilon$-saturated classes containing a given $\mch$, as we can increase 
the class before taking the saturation.} 
\end{lemma}

\begin{proof}
By induction on $n<\omega$ we build an increasing sequence of functions $\langle F_n : n < \omega \rangle$ such that: 
\begin{enumerate}
    \item[(a)] for $n<\omega$, $F_n$ is an injective map from $\mch_n(\epsilon)$ into $\mbh$, 
    \item[(b)] for $n<m < \omega$, $F_m \rstr \mch_n(\epsilon) = F_n$, and
    \item[(c)] for $n<\omega$, for every $x \in X$ and 
$h \in \dom(F_n)$, $h(x) = 1$ if and only if 
$(F_n(h))(x) = 1$.
\end{enumerate}
[Since we identify two functions which are the same on $X$, condition (c) amounts to saying the map will be the identity.]  
For $n=0$, let $F_0$ be the identity as $\mch \subseteq \mbh$. Consider $n+1$. Enumerate the elements of $\mch_{n+1}(\epsilon) \setminus \mch_n$ by the appropriate cardinal $\alpha$ as $\langle h_\beta : \beta < \alpha \rangle$. The functions 
$h_\beta$ in this list are necessarily distinct functions, and, 
by definition of $\mch_{n+1}$, they are distinct from the elements of $\mch_n$.  
It suffices to specify 
$F_{n+1}$ on each $h_\beta$ individually, so fix $h_\beta$. 
By definition of $\mch_{n+1}(\eps)$, we have $h_\beta = \trv_H$ for some finite weighted $\epsilon$-good set  
$H = \{ (g_i, \gamma_i) : i < k \}$ where the $g_i$ belong to $\mch_n(\eps)$, but $\trv_H \notin \mch_n$. By condition (c), $Y = \{ (F_n(g_i), \gamma_i) : i < k \}$ is also a weighted $\epsilon$-good set, where the $F_n(g_i)$ belong to $\mbh$, and $Y$ represents the same function on $X$ as $H$. So $g = \trv_H = \trv_Y$ is realized in the class $\{ \trv_Y : Y $ is an $\epsilon$-hypothesis of $\mbh$ $\}$, 
hence it is a fortiori finitely realized in this class,
hence $g$ is itself a type of $\mbh$, 
hence it is realized by some (and indeed exactly one) $f \in \mbh$ by the assumption of saturation. Note that $f \notin \operatorname{range}(F_n)$ because we assumed that $\trv_H \notin \mch_n$ and $F_n$ satisfied condition (c). 
Set $F_{n+1}(h_\beta) = f$.   
\end{proof}

\section{Intermediate $\epsilon$} \label{s:intermediate-epsilon}

We continue with the assumption that $\mch$ is a Littlestone class, of $\ldim$ $\ell$ and $\vc$ dimension $d$. In \ref{sat:hypothesis} we saw that if $\epsilon$ is sufficiently small as a function of $d$, then the $\vc$-dimension of $\mch_\infty$ does not change, and if 
$\epsilon$ is sufficiently small as a function of $\ell$, then the $\ldim$ of $\mch_\infty$ also does not change.  
In this section we prove the a priori quite surprising theorem that if $\epsilon$ is bounded as a function of $d$ (here we do not refer to  $\ell$) then $\mch_\infty$ remains a Littlestone class possibly 
of larger $\ldim$. An upper bound on this larger $\ldim$ can be obtained from the proof as we will discuss at the end of the section. 

In the next theorem, for definiteness we could have said ``For any $C \geq 13$'' but we give an approximate value, whose 
calculation is justified in \ref{C:approx}.  

\begin{theorem}\label{thm:littlestone-Cd}
    There exists a constant $C\approx 12.0412$ that satisfies the following. 
Let $\mch$ be a Littlestone class, with $\ldim(\mch)=\ell$ and $\vc dim(\mch) = d$. If $0<\eps<\frac{1}{Cd}$, then $\mch_{\infty}(\eps)$ is Littlestone. 

    More precisely: given $d$, we are free to choose:

    \begin{enumerate}
    \item any constant $\mathbf{C} \geq C$
    \item based on $\mathbf{C}$, any $\epsilon_1 \in (0,1/2)_{\mathbb{R}}$ 
    such that $\mathbf{C}d>20\cdot(d/\eps_{1})\cdot\log(1/\eps_{1})$
    \item based on $\eps_{1}$, any $\eps_{2}\in(0,1/2)_{\mathbb{R}}$ such that $\eps_{1}<\eps_{2}$
    \item based on $\eps_{2}$, any $\eps_{3}\in(0,1/2 - \eps_{2})_{\mathbb{R}}$

    \end{enumerate}
Then, the Littlestone dimension of the class $\mch_{\infty}(\eps)$ is upper bounded in terms of 
$d/(\eps_{3})^{2}$ and the original Littlestone dimension of $\mch$, i.e. $\ell$.
\end{theorem}

In the statement, ``20'' in item (2) is an upper bound for the constant in the $\epsilon$-net theorem, and the function bounding the final Littlestone dimension depends on the first part of the proof of Theorem 3 in \cite{alon20} (bounding Littlestone dimension of fixed finite boolean combinations of elements of a Littlestone class) as discussed at the end of the section. 

Before giving the proof, we need to prove several claims. 
First, we show that the functions in $\mch_{\infty}(\eps)$ are all $k$-realizable, as long as $k$ is not too large in terms of $\eps$. 
(Compare \ref{sat:hypothesis} and Example \ref{example12} below.) 
\begin{claim}\label{claim:k-realizability-backtrack}
    Let $k<\lfloor 1/\eps\rfloor$. Then 
    $\mch_{\infty}(\eps)\subseteq\mch^{\kr}$.
\end{claim}
\begin{proof}
The claim asserts that any $h \in \mch_\infty(\epsilon)$ is already 
$k$-realizable in $\mch$. If $h \in \mch_\infty(\epsilon)$ then 
we can choose a minimal $m$ so that $h \in \mch_m(\epsilon)$. 
Clearly $h$ is $k$-realizable in $\mch_m$, so if $m=0$ we are done. 
Otherwise, suppose for a contradiction that $m = n+1$ and 
$h$ is $k$-realizable in $\mch_{n+1}$ but not in $\mch_n$.   

As $h \in \mch_{n+1}$ there is a weighted $\epsilon$-good set $H = \{ (g_i, \gamma_i) : i < s \}$
which represents $h$, hence for each $x \in X$, $\sum \{ \gamma_i : i<s, g_i(x) \neq h(x) \} < \epsilon$. 
Hence given any $x_0, \dots, x_{k-1}$ from $X$, 
we have that $\sum \{ \gamma_i : i < s, j<k, g_i(x_j) \neq h(x_j) \} < \epsilon k \leq 1$ (by choice 
of $\epsilon$). 
So there exists $(g,\gamma) \in H$ such that $g(x_j) = h(x_j)$ for $j=0, \dots, k-1$. 
By definition $g \in \mch_n$. This shows $h$ is $k$-realizable in $\mch_n$, contradicting our 
choice of $n$.   
\end{proof}

In the following claim, we fix $C$ and $\epsilon$ from Theorem \ref{thm:littlestone-Cd}. Note that this background $\epsilon$ is not used in this claim. 

The claim analyzes a single function $f$ from $\mch^\kr$ in the case where $k = Cd$.  Recall that a priori we do not know the complexity of $\mch^{\Cdr}$, but we do know that $\mch$ is Littlestone and indeed of $\vc$ dimension $d$. Let $\mch_f$ denote the class of symmetric differences of $f$ with each element of $\mch$, which remains a $\vc$ class of dimension $d$. Given any measure (we will only use finitely supported probability measures) $\mu$ on $X$ and $h_f \in \mch$, let $\mu(h_f)$ mean $\mu ( \{ x \in X : h_f(x) = 1 \})$. There are two key ideas. First, since $\mch_f$ is a $\vc$ class, the $\epsilon$-net theorem [note: we should call it the $\epsilon_1$-net theorem since we are applying it to an $\epsilon_1$ we shall choose presently and not our background $\epsilon$] applies: for some $M = M(\epsilon_1, d)$, for any measure $\mu$ on $X$ there is $S \subseteq X$, $|S| \leq M$ which ``sees'' all nontrivial hypotheses (i.e. for every $h_f$ of measure $\geq \epsilon_1$ there exists $s \in S$ so that $h_f(s) =1$). Since $\mch_f$ measures symmetric difference with $f$, this amounts to saying that if $h \in \mch$ differs from $f$ by more than $\epsilon_1$, the set $S$ sees it. 
Second, note this along with Claim \ref{claim:k-realizability-backtrack} tells us that for any choice of $\epsilon_1$ satisfying Theorem \ref{thm:littlestone-Cd}(2), since $Cd > M(\epsilon_1,d)$, when $\mu$ is any finitely supported probability measure on $X$ and $S$ is an $\epsilon_1$-net,  $Cd > |S|$ hence the assumption that $f$ is $Cd$-realizable in $\mch$ means that there is some $h \in \mch$ so that $h$ and $f$ agree on $S$, \emph{hence} the global difference between $h$ and $f$ is $<\epsilon_1$. 
\begin{claim}\label{claim:epsilon-net-one-hyp}
    Let $\mathcal{D}$ be a distribution over $X$ with finite support, and let $\mch$ be [a Littlestone class] of $\vc dim(\mch)=d$. Let $C$ be as in Theorem \ref{thm:littlestone-Cd}. If $\eps_{1}\in(0,1/2)_{\mathbb{R}}$ satisfies the inequality $C>20\cdot\eps_{1}^{-1}\cdot \log(\eps^{-1}_{1})$ then for every $f \in \mch^{\Cdr}$, there exists $h\in\mch$ such that
    \[ Pr_{x\sim\mathcal{D}} [h(x)\neq f(x)] < \eps_{1}.\]
\end{claim}
\begin{proof}
        Define $\mch_{f}:=\{h_{f}: h\in\mch\}$, where for each $h\in\mch$ and $x\in X$ we define $h_{f}(x):=\delta_{f(x)\neq h(x)}$, as the symmetric difference of $h$ and $f$. Since the map $h\mapsto h_{f}$ is an involution, shattered sets are the same for both classes, and $VCdim(\mch)=VCdim(\mch_{f})=d$.\\\\
        Then, by $\eps$-net Theorem on $\mch_{f}$ (see e.g. Theorem 10.2.4 in \cite{matousek-lectures-discgeo}), there exists a constant $C_{net}<20$, and a number $M:=m(d,\eps_{1})\leq C_{\text{net}}\cdot d\cdot(1/\eps_{1})\cdot\log(1/\eps_{1})$ such that a sample $\mathcal{S}$ of size $M$ from $X\sim\mathcal{D}$, satisfies with positive probability that:
            \[\forall h_{f}\in\mch_{f},\hspace{0.5cm} Pr_{x\sim\mathcal{D}}[h_{f}(x)=1]\geq\eps_{1} \hspace{0.2cm}\Rightarrow\hspace{0.2cm} \exists s\in \mathcal{S}: h_{f}(s)=1.\]
        Following the correspondence between $\mch$ and $\mch_{f}$, the above is equivalent to
          \[\forall h\in\mathcal{H},\hspace{0.5cm} Pr_{x\sim\mathcal{D}}[h(x)\neq f(x)]\geq\eps_{1} \hspace{0.2cm}\Rightarrow\hspace{0.2cm} \exists s\in \mathcal{S}: h(s)\neq f(s).\]
        The condition on $\eps_{1}$ implies $M\leq Cd$. Let $\mathcal{S}$ be a sample with the above property. Since $|\mathcal{S}|\leq M \leq Cd$, then $f\in\mch^{\Cdr}$ implies there is some $h\in\mathcal{H}$ such that for all $x\in \mathcal{S}$, $f(x)=h(x)$. By contrapositive, then, $Pr_{x\sim\mathcal{D}}[h(x)\neq f(x)] < \eps_{1}$.
        \end{proof}

We continue to analyze a given $f \in \mch^{\Cdr}$. The next Claim applies Littlestone minimax to conclude that any such $f$ can be represented by a weighted $\epsilon_{2}$-majority vote of elements of $\mch$, for any $\eps_{2}\in(\eps_{1},1/2)_{\mathbb{R}}$. Why? Condition (A) of the Littlestone Minimax, Theorem A above, is exactly what is guaranteed by Claim \ref{claim:epsilon-net-one-hyp}. 

\begin{claim}\label{claim:eps-hypothesis-via-littlestone-minimax}
    Let $\mathcal{D}, \mch, d, C, \eps_{1}$ be defined as before. In particular, recall that $\mch$ is a Littlestone class with $VCdim(\mch)=d$.  
Fix $f \in \mch^{\Cdr}$.  
For any $\eps_{2}\in(0,1/2)$ such that $\eps_{2}>\eps_{1}$, there exists a finite weighted sequence $\{(h_{i}, \gamma_{i}) \;:\; i<n \}$ in $\mch$ such that, for every $x\in X$:
            \begin{enumerate}[label=(\roman*)]
            \item there is an $\eps_{2}$-majority vote $\trv(x)\in\{0,1\}$, i.e. a value  for which $$\sum\{\gamma_{i}:h_{i}(x)= 1-\trv(x), i<n\}<\eps_{2};$$
            \item $\trv(x) = f(x)$. 
            \end{enumerate}
    \end{claim}
    \begin{proof}
Suppose we are given $\{ (x_t, w_t) : t < s \}$ as in condition (A) of Littlestone Minimax. By \ref{claim:epsilon-net-one-hyp} there is some 
$h \in \mch$ whose weighted error against $f$ is $<\epsilon_1$. So by condition (B) of Littlestone Minimax there is an 
$\epsilon_2$-hypothesis $H = \{ h_i, \gamma_i : i < n \}$ of $H$ so that $\trv_H = f$. 
        \end{proof}

We have seen that our $f \in \mch^{\Cdr}$ is expressible as a weighted $\epsilon_2$-majority vote of some $H \subseteq \mch^{\Cdr}$. $H$ is finite, but a priori we have no bound on its size. Our last claim shows that we may assume a uniform finite bound on the size of $H$, as a function of $d$ and any fixed $\eps_{3}\in(0,1/2-\eps_{2})_{\mathbb{R}}$. (Why? Use the fact that the dual of any VC class is also a VC class, take the finitely supported measure on $\mch$ corresponding to the weighted $\epsilon_2$-good set $H$, and choose a representative sample from it of size about $d/(\epsilon_3)^2$, as the proof describes.)  
   
\begin{defn} \label{d:majority} Given a hypothesis class $\mch$ and an odd $p \in \mathbb{N}$, let $\langle g_j : j < p \rangle$ be a sequence of elements of $\mch$, 
possibly with repetition. 
The \emph{majority vote} of this sequence is defined to be the function $g : X \rightarrow \{ 0, 1 \}$ given by: $g(x) = 1$ if $| \{ j < p : g_j(x) = 1 \} | > (p-1)/2$ 
and $g(x) = 0$ otherwise.  
Let 
\[ \Maj_p(\mch) \] 
be the set of functions arising as majority votes of sequences $\langle g_j : j < p \rangle$ in the $p$-fold Cartesian power $ \mch \times \cdots \times \mch$.  
\end{defn}

In particular, definition of majority in \ref{d:majority} is the counting majority and is not directly related to a background epsilon. 
 
    \begin{claim} \label{claim-vc-uniform}
    Let $\mathcal{D}, \mch, d, C, \eps_{2}$ be defined as before. For some odd $p\in\mathbb{N}$, for every $f\in\mch^{\Cdr}$, there exists 
a sequence $\langle f_i : i < p \rangle$ of elements of $\mch$, possibly with repetition, whose majority vote is $f$. In other words, $\mch^{\Cdr} \subseteq \Maj_p(\mch)$.  
    \end{claim}
    \begin{proof}
Let $f \in \mch^{\Cdr}$ be given, and let $H = \{ (h_i, \gamma_i) : i < n \}$ be the corresponding $\epsilon_2$-hypothesis given by 
\ref{claim:eps-hypothesis-via-littlestone-minimax}. Recall that $\mch_* := \dual(\mch)$ (\ref{d:dual}) is also a VC class, of VC dimension $d_* \leq 2^{d+1}-1$. 
So it satisfies uniform convergence: namely, there is a constant $c_* = c_*(\epsilon_3, d_*)$ and a sample size $m_* = m_*(\epsilon_3, d_*) \leq c_* d_* (\epsilon_3)^{-2}$ with the following 
property.  
For any probability measure on $\mch$ [recall that in the dual $\mch_*$, the domain is $\mch$], and in particular for 
that corresponding to $H = \{ (h_i, \gamma_i) : i < m \}$,  
for any $p \geq m_*$ [for definiteness, for us, let $p$ be the smallest odd number $\geq m_*$], 
with high probability an iid sample $\langle g_t : t < p \rangle$ satisfies: for 
every $x \in X$, the ``empirical and actual measures are close,'' meaning that 
\[  \left| \left( \frac{ | \{ t < p : g_t(x) = 1 \}| }{|p|} \right) - \left( \sum \{ \gamma_i : i < n, h_i(x) = 1 \} \right)   \right| < \epsilon_3.     \] 
(Informally, each $x \in \mch_*$ determines the set of $h \in \mch$ which label it by 1; the size of that set according to $H$ should be 
close to the size of that set according to the sample, which is of course sub-sampled from $\{ h_i : i < n \}$ by our choice of measure.) 

Since this happens with high probability, in particular we can fix \emph{some} sequence $\langle g_t : t < p \rangle$ on which this happens. 
Now $\sum \{ \gamma_i : i < n, h_i(x) = 1 \}$ is either $<\epsilon_2$ or $\geq 1-\epsilon_2$, so 
$\frac{ | \{ t < p : g_t(x) = 1 \}| }{|p|}$ is either $<\epsilon_2 + \epsilon_3$ or $\geq 1-\epsilon_2-\epsilon_3$. In other words, since 
$(\epsilon_2 + \epsilon_3)< 1/2$, we have that for every $x$, $\langle g_t(x) : t < p \rangle$ has a numerical majority which agrees with $H$ and 
hence represents $f$. 
\end{proof}

    \sbr
    \noindent 
    Now we can conclude by applying known bounds on the $\ldim$ of fixed finite boolean combinations of elements of a Littlestone class. We will use:

\begin{fact}[\cite{alon20}, Theorem 3] \label{fact:maj-votes} If $\mch$ is Littlestone then for any fixed odd $p \in \mathbb{N}$, 
$\Maj_p(\mch)$ is also Littlestone. 
 Concretely, we have the following upper bound:
    \[\ldim(\text{Maj}(\mch,\ldots,\mch))\leq \tilde{O}(p^{2}\cdot \ldim(\mch)).\]  
\end{fact}
The proof of this fact uses a technique named Online boosting. In model theoretic language, the familiar fact is that if $\varphi$ is a stable formula, then any fixed finite boolean combination 
of $\varphi$'s is also stable, though of a possibly larger rank.

We now formally assemble the pieces to prove Theorem \ref{thm:littlestone-Cd}:
    \begin{proof}[Proof of Theorem \ref{thm:littlestone-Cd}]
By Observation \ref{claim:k-realizability-backtrack}, regardless of what our constant $c$ is, $\eps<1/cd$ implies that $\mch_{\infty}(\eps)\subseteq\mch^{\cdr}$. 
Also, as long as $cd > d+1$, Theorem \ref{sat:hypothesis}(4), $\vc(\mch_\infty(\eps)) = \vc(\mch) = d < \infty$.       

For the rest of the proof, fix $C$ satisfying the statement of the theorem (recall $C\approx12.0412$) 
and choose some $\eps_{1}>0$ satisfying $C>20\cdot\eps_{1}^{-1}\cdot \log(\eps^{-1}_{1})$. 
Then Claim \ref{claim:epsilon-net-one-hyp} will allow us to find, for every $f\in\mch^{\Cdr}$, 
some $h\in\mch$ with error less than $\eps_{1}$ with respect to $f$, by realizability and the $\eps$-net property.  
For any $\eps_{2}>\eps_{1}$, Claim \ref{claim:eps-hypothesis-via-littlestone-minimax} yields an $\eps_{2}$-hypothesis $H$ on $\mch$ such that $\trv_{H}$ equals $f(x)$, for all $x\in X$. 
Then, Claim \ref{claim-vc-uniform} allows us to find a uniform size for these representations: 
it shows there is $p\in\mathbb{N}$ so that, sampling $p$ times from the distribution given by $H$, and then taking the $p$-wise 
majority vote, we obtain $f(x)$ for all $x\in X$.

We have shown that under the assumptions of the theorem  
\[ \mch_\infty(\eps) \subseteq \mch^{\Cdr} \subseteq \Maj_p(\mch) \]
Concluding by Fact \ref{fact:maj-votes}, $\mathcal{H}_{\infty}(\eps)$ is Littlestone since $\Maj_p(\mch)$ is.  
    \end{proof}

\begin{rmk} \label{C:approx} 
    The choice $C\approx12.0412$ comes from the following considerations:
    \begin{itemize}
        \item The decreasing and injective behavior of $f(x)=\frac{1}{x}\cdot\log(\frac{1}{x})$ over $(0,1/2)_{\mathbb{R}}$.
        \item The argument in Claim \ref{claim:epsilon-net-one-hyp} needs the inequality $Cd>d\cdot f(\eps_{1})$ to hold
        \item Continuity of $f(x)$ at $1/2$, and $12.0412\approx\lim_{x\to1/2^{-}}f(x)$. 
    \end{itemize}
\end{rmk}

	\sbr

	\section{Conclusions for Littlestone classes}

	In this section, we summarize the main results of this paper so far into the following theorem, also visible in Figure \ref{fig:phase-diagram}.
	\begin{theorem}
	    Let $(X,\mch)$ be a hypothesis class such that $\ldim(\mch)=\ell$ and $\vc(\mch)=d$. The saturations of these classes satisfy the following:
	    \begin{enumerate}[label=(\alph*)]
		\item  When $\eps \leq \frac{1}{\ell+1}$, $\ldim(\mch_{\infty}(\eps))=\ldim(\mch)=\ell$;
        \item When $\eps \leq \frac{1}{d+1}$, $\vc(\mch_\infty(\epsilon)) = \vc(\mch) = d$. 
		\item For $\eps<\frac{1}{Cd}$, for some $C\gtrapprox 12.0412$, $\ldim(\mch_{\infty}(\eps))<\infty$ although it need not stay the same as $\ldim(\mch)$;
		\item For $\eps>\frac{1}{d+1}$, it is possible for both $\ldim(\mch_{\infty})$ and $\vc(\mch_{\infty})$ to be infinite.
	    \end{enumerate}
	\end{theorem}

	\begin{figure}[ht]\label{fig:phase-diagram}
	    \centering
	    \includegraphics[width=0.9\linewidth]{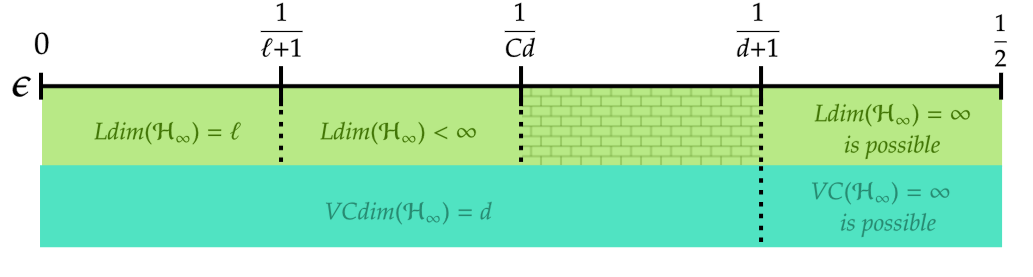}
	    \caption{Regimes of $\eps\in(0,\frac{1}{2})_{\mathbb{R}}$ and the corresponding statements of what we showed about $\ldim(\mch_{\infty})$ and $\vc(\mch_{\infty})$, given that $\ldim(\mch)=\ell$ and $\vc(\mch)=d$. Here $C$ is from Theorem \ref{thm:littlestone-Cd}}
	    \label{fig:diagram-summary}
	\end{figure}

	\begin{proof}
	    Parts $(a)$ and $(b)$ are both shown in Theorem \ref{sat:hypothesis}. Part $(c)$ is shown in Theorem \ref{thm:littlestone-Cd}. Part $(d)$ is shown by Example \ref{eg:counterex-vcblowsup}.
	\end{proof}

	\br

	\section{Definition of $\epsilon$-saturation for $\gee$} \label{s:graphs} 

	We now give the definitions in the case of graphs. A crucial role will be played by excellent sets. Remember these can be different from good sets; see \S \ref{example:good-excellent} .

	\begin{conv}
	When dealing with graphs we will always assume $\epsilon < 1/4$.  
	\end{conv}

	\begin{conv}
	$\gee = (V,E)$, which we may abbreviate $\gee$, will usually denote a stable graph 
	with:
	\begin{itemize}
	\item finite edge stability, also called threshold dimension, $s$,  
	\item hence finite Littlestone dimension $\ell$,  
	\item hence finite VC dimension $d$. 
	\end{itemize} The vertex set $V$ can be finite, but need not be. 
	We allow self-loops, that is, we allow $(v,v)$ as an edge for any vertex $v$.\footnote{Briefly, 
	this is because we will add $\epsilon$-excellent sets of vertices as new ``virtual'' vertices and 
	any such set $A$ will have a well-defined edge or nonedge to any other such set, including itself. 
	There is a priori no reason to erase the information of which of the two it is.}  
	\end{conv}

	As before, we will repeat the assumption of stability in the main definitions and 
	theorems which use it. We will often write ``$v \in \gee$'' to mean 
	$v \in V(\gee)$, and ``$v \sim w$'' to mean $(v,w) \in E(\gee)$.

	Compared to the definition of 
	$\epsilon$-hypothesis, the 
	definition of $\epsilon$-vertex will require an extra step due to the fact that 
	we would like edges between $\epsilon$-vertices to be well defined (this corresponds 
	to $\epsilon$-excellent in the stable regularity lemma). As before, we will use weighted good and excellent sets.  

	\begin{defn} Given $\epsilon$ and $\gee$, 
	a \emph{weighted $\epsilon$-good set} 
	of $\gee$ is a finite set of pairs:
	\[ A = \{ (v_i, \gamma_i) : i < k \} \]
	where $k \in \mathbb{N}$, and:
	\begin{enumerate}
	\item $\{ v_i : i < k \} \subseteq \gee$, and
	$\langle v_i : i < k \rangle$ is without repetition.
	\item each $\gamma_i \in (0,1]_{\mathbb{R}}$, and $\sum \{ \gamma_i : i < k \} = 1$.
	\item $($$\epsilon$-goodness$)$ For each $v \in V$,  either  
	\[  \sum \{ \gamma_i : v_i \sim v \} \geq 1-\epsilon 
	\mbox{ or } \sum \{ \gamma_i : v_i \not\sim v \} \geq 1-\epsilon.  \]
	In the first case, write $\trv(A,v) = 1$, and in the second $\trv(A,v) = 0$.
	\end{enumerate}
	\end{defn}

	\begin{defn} \label{d:epsilon-vertex} 
	Given $\epsilon$ and $\gee$,  
	an $\epsilon$-vertex $($or a \emph{weighted $\epsilon$-excellent set}$)$ of $\gee$ is 
	a finite set of pairs:
	\[ B = \{ (v_i, \gamma_i) : i < k \} \]
	where $k \in \mathbb{N}$, $B$ is a weighted $\epsilon$-good set, and moreover:
	\begin{enumerate}
	\item[(4)] $($$\epsilon$-excellence$)$ For every weighted $\epsilon$-good set $A$, 
	either  
	\[  \sum \{ \gamma_i : \trv(A,v_i) = 1 \} \geq 1-\epsilon   
	\mbox{ or } \sum \{ \gamma_i : \trv(A, v_i) = 0 \} \geq 1-\epsilon.  \]
	In the first case, write $\trv(B, A) = 1$, and in the second $\trv(B,A) = 0$. 
	\end{enumerate}
	\end{defn}

	That \ref{d:epsilon-vertex} is sufficiently symmetric for our purposes 
	is justified by the following (which also explains our convention 
	that in working with graphs, $\epsilon < 1/4 < 2-\sqrt{3}$).

	\begin{rmk} \label{ewd}
	    Let $\eps\in(0,2-\sqrt{3})$, and let $A$ and $B$ be two $\eps$-excellent weighted sets of $G$. Then, the values $\trv(A,B)$ and $\trv(B,A)$ are well-defined and equal.  
	\end{rmk}
	\begin{proof}
        Suppose for a contradiction that $\trv(A,B)=1$ and $\trv(B,A)=0$. This implies that, if the weight for each $a_{i}$ is $\gamma_{i}$ and the weight for each $b_{i}$ is $\rho_{i}$, then:
        \[\sum\{\gamma_{i}:\trv(a_{i},B)=1\}\geq 1-\eps,\hspace{1cm} \sum\{\rho_{i}:\trv(A,b_{i})=1\}\leq \eps\] 
        The first bound above implies that the total sum of the products $\gamma_{i}\rho_{j}$ over all edges from $A$ to $B$ is lower bounded by
        \[\sum_{a_{i}\in A:\trv(a_{i},B)=1}\gamma_{i}\cdot (1-\eps) + \sum_{a_{i}\in A:\trv(a_{i},B)=0}\gamma_{i}\cdot 0 \geq (1-\eps)^{2}.\]
        The second bound above implies that the same total sum of products of weights described above is upper bounded by
        \[\sum_{b_{i}\in B:\trv(A,b_{i})=1}\rho_{i}\cdot 1 + \sum_{b_{i}\in B:\trv(A,b_{i})=0}\rho_{i}\cdot \eps \leq 2\eps.\]
	    The symmetry of edges in $G$ and these two bounds imply that
        \begin{align*}
		&(1-\eps)^{2} < \sum_{(a_{i},b_{j})\in 
        E(A,B)}\gamma_{i}\rho_{j} < 2\eps.\\
		\Rightarrow& \eps\in(2-\sqrt{3}, 2+\sqrt{3})
	    \end{align*}
	   These values for $\eps$ are ruled out by the hypothesis that $\eps < 2-\sqrt{3}$.  
	\end{proof}

	\begin{defn}
	Given $\epsilon$ and $\gee$, whenever $A, B$ are $\epsilon$-vertices, 
	say there is an $\epsilon$-edge between $A,B$, in symbols 
	$A \sim B$ or $A \sim_\epsilon B$, to mean that $\trv(A,B) = \trv(B,A) = 1$, 
	and say there is an $\epsilon$-nonedge, in symbols $A \not\sim B$ or 
	$A \not\sim_\epsilon B$, to mean that $\trv(A,B) = \trv(B,A) = 0$.  
	\end{defn}

	\begin{ntn}
	Whenever $A$ is an $\epsilon$-vertex, let $\trv_A$ be the function whose 
	domain is the set of $\epsilon$-vertices of $\gee$ and whose range is $\{ 0, 1 \}$, 
	given by $B \mapsto \trv(A,B) = \trv(B,A)$.  
	\end{ntn}

	\begin{conv}
	Since for any $v \in V$, $\{ v \}$ is an $\epsilon$-vertex, 
	we may drop parentheses, writing $v$ for $\{v\}$ or $w$ for $\{w\}$ in 
	expressions like $\trv$ or $\sim_\epsilon$. 
	\end{conv}

	To define the $\epsilon$-saturation, we will formally extend our graphs  
	by adding $\epsilon$-vertices, and 
	so all our edges are a priori $\epsilon$-edges. 
	Note, however, that in Definition \ref{d:epsilon-saturation-H} in some sense 
	we constantly forget this: the graphs $\gee_n$ 
	at each stage in the construction are \emph{actual graphs} (which come from 
	formally adding $\epsilon$-vertices as actual new vertices, then   
	``rounding up'' $\sim_\epsilon$ to $\sim$ and ``rounding down'' 
	$\not\sim_\epsilon$ to $\not\sim$). Since there will often be 
	many $\epsilon$-vertices with the same $\sim_\epsilon$-neighborhood, 
	in order to avoid triviality, we only add one representative of each 
	such neighborhood class.\footnote{ Indeed, since $\trv(A,B) = \trv(A^\prime, B^\prime)$ 
	whenever $A, A^\prime$ and $B, B^\prime$ have the same $\sim_\epsilon$-neighborhoods, 
	when we will want to remember these arose from 
	$\epsilon$-vertices, we are free to calculate by choosing some 
	possibly different representative of the class. } This is the analogue of 
not repeating functions in the case of $\mch$ above.  

	\begin{defn}[$\epsilon$-types] \label{epsilon-type-G}
	Given $\gee = (V,E)$ and a function $\trv: V \rightarrow \{ 0, 1 \}$: 
	\begin{enumerate}
	\item $\trv$ is $k$-realized in $\gee$, for some given $k \in \mathbb{N}$, if for any 
	$W \subseteq V$, $|W| \leq k$ there is $v \in V$ such that 
	$\trv_{\{(v,1)\}} \rstr W = \trv \rstr W$. 

	\item $\trv$ is finitely realized in $\gee$ if it is $k$-realized in $\gee$ for all $k \in \mathbb{N}$.   

	\item Given also $\epsilon < 1/4$, an \emph{$\epsilon$-type} for $\gee$ 
	is a function $\trv: V \rightarrow \{ 0, 1 \}$  
	such that for any finite $W \subseteq V$, there is an $\epsilon$-vertex $A$ of $\gee$ 
	such that\footnote{informally, we could say the function is  
	finitely realized by $\{ \trv_A : A \mbox{ an $\epsilon$-vertex of $\gee$} \}$. }  
	\[ \trv_A \rstr W = \trv \rstr W. \]
	\end{enumerate}
	\end{defn}

	The central definition of this paper, in the case of graphs, is:  

	\begin{defn}[The $\epsilon$-saturation of $\gee$] \label{d:epsilon-saturation-G} 
	Suppose we are given a graph $\gee = (V,E)$ and $0 < \epsilon < 1/4$. 
	Define by induction on $n\leq \omega$ a graph $\gee_n(\epsilon) = (V_n, E_n)$: 

	\begin{enumerate}
	\item $\gee_0(\epsilon) = \gee = (V,E)$.
	\item $\gee_{n+1}(\epsilon) = (V_{n+1}, E_{n+1})$ where:\footnote{Why do we not  
	ask that for every $\epsilon$-vertex $A$ of $V_n$, there is precisely one $\epsilon$-vertex  
	$B$ of $V_n$ such that $B \in V_n \cup W$ and $\trv_B = \trv_A$? This would impose a  
	retroactive restriction on $G_n$ saying that there cannot be two singletons with the same  
	neighborhood. There is a priori no reason this should be true of $G_0$, that is,  
	we do not require $G_0$ to be prime, even though we take care not to add instances of  
	non-primality later on.}
	\begin{enumerate}
	\item $V_{n+1} =  V_n \cup W$ where $W$ is a representative set 
	of $\epsilon$-vertices of $V_n$ 
	meaning: for every $\epsilon$-vertex $A$ of $V_n$ with $|A| > 1$, 
	there is precisely one $\epsilon$-vertex $B$ of $V_n$ such that $B \in W$ and $\trv_B = \trv_A$. 

	\item $E_{n+1}$ is given by:
	\begin{enumerate}
	\item for $(v,w) \in V_n \times V_n$, 
	$(v,w) \in E_{n+1}$ iff $(v,w) \in E_n$; 
	\item otherwise, for $(v,w) \in (V_{n+1} \times V_{n+1}) \setminus (V_n \times V_n)$, 
	\\ $(v,w) \in E_{n+1}$ iff $v \sim_\epsilon w$. 
	\end{enumerate}
	\end{enumerate}

	\item $\gee_\infty(\epsilon) := \mch_\omega(\epsilon) = (V_\infty, E_\infty)$ where 
	$V_\infty = V_\omega = \bigcup \{ V_n : n < \omega \}$ and 
	$E_\infty = E_\omega = \bigcup \{ E_n : n < \omega \}$.  
	\end{enumerate} 
	We call $\gee_\infty(\epsilon)$ ``the $\epsilon$-saturation of $\gee$.''
	\end{defn}

	Here also the name, and the connection to $\epsilon$-types, will be justified in due course.
	Also, we could have continued the definition by transfinite induction, but this is not used here.

\section{Characterization of $\epsilon$-saturation for stable graphs}

	The following definition and lemma have a natural analogue for hypotheses classes, which we do not 
	state separately.  Note that in the case of a hypothesis class $(X, \mch)$, the domain $X$ did not change, we simply 
added new functions; whereas here the domain is a priori constantly changing, so things have a slightly different character.  

	\begin{defn}
	A \emph{partial} $\epsilon$-type of a graph or hypothesis class means that we allow the function 
	$\trv$ to be partial. 
	\end{defn}

	\begin{lemma}[Types can be extended] \label{extending-types}
	Let $\bfs$ be a partial $\epsilon$-type of a graph $G$. Then
	$\bfs$ can be extended to an $\epsilon$-type, that is, there is an $\epsilon$-type $\trv \supseteq \bfs$. 
	In fact, $\bfs$ can be extended to $V \cup \{ B : B $ is a weighted $\epsilon$-good subset of $G$ $\}$. 
	\end{lemma}

	\begin{proof}
	It will be notationally convenient to identify functions with their graphs.  
	Enumerate the elements of 
	\[ \mbox{ $V \cup \{ B : B $ is a weighted $\epsilon$-good subset of $G \} \setminus \dom(\bfs)$ } \]  
	by the cardinal $\kappa \geq |G|$)  as
	\[ \langle c_\alpha : \alpha < \kappa \rangle.  \]
	Let us show by induction on $\alpha < \kappa$ that we can build an increasing
	continuous chain of partial $\epsilon$-types $\trv_\alpha$ such that:
	\begin{enumerate}
	\item $\trv_0 = \bfs$
	\item each $\trv_\alpha$ is an $\epsilon$-type, i.e. finitely realized by $\epsilon$-vertices  
	\item $\beta < \alpha \implies \trv_\beta \subseteq \trv_\alpha$
	\item if $\alpha$ is a limit ordinal then $\trv_\alpha = \bigcup \{ \trv_\beta : \beta < \alpha \}$ 
	\item if $\alpha = \beta+1$ then $c_\beta \in \dom(\trv_\alpha)$.
	\end{enumerate}
	For $\alpha = 0$ and $\alpha$ limit this is easy as being finitely realized by 
	$\epsilon$-vertices is preserved under unions of chains.

	Let $\alpha = \beta + 1$. If $c_\beta \in \dom(\trv_\beta)$, set 
	$\trv_\alpha = \trv_\beta$ and we are done. 
	Otherwise, we ask if $\trv_\beta \cup \{ (c_\beta, 1) \}$ is finitely realizable 
	by $\epsilon$-vertices.  
	If not, there is a finite $\trv^\prime \subseteq \trv_\beta$ so that
	$\trv^\prime \cup \{ (c_\beta, 1) \}$ is not realized by any $\epsilon$-vertex of $\gee$.  
	Likewise if $\trv_\beta \cup \{ (c_\beta, 0) \}$ is not finitely realizable by 
	$\epsilon$-vertices, there is a finite $\trv^{\prime\prime} \subseteq \trv_\beta$ such that 
	$\trv^{\prime\prime} \cup \{ (c_\beta, 0) \}$ is not finitely realizable by $\epsilon$-vertices. 
	But then $\trv^\prime \cup \trv^{\prime\prime}$ cannot be finitely realized 
	by $\epsilon$-vertices, because
	if it were, say by $A$, this same $A$ would have to take some value on $c_\beta$
	which would contradict one of our two assumptions.  
	So one of the two additions must be possible, completing this step and hence the proof. 
	\end{proof}

	\begin{theorem} \label{t:saturation2} 
	Let $\gee$ be a $k$-stable graph, $\epsilon < 1/4$ and $\gee_\infty(\epsilon)$ its 
	saturation. Suppose $\gee_\infty(\epsilon)$ is also Littlestone. Then 
	$\gee_\infty(\epsilon)$ realizes all of its $\epsilon$-types. 
	\end{theorem}

	The proof will depend on several claims, slightly different than Theorem  \ref{t:saturation1}. 

\begin{claim} \label{o:easier} 
If $V = \{ (v_i, \gamma_i) : i < m \}$ is an $\epsilon$-vertex of $G_\infty(\epsilon)$ then  
it is also an $\epsilon$-vertex of $G_n(\epsilon)$ for some $n < \omega$. 
\end{claim}

\begin{proof} 
Let $n < \omega$ be such that $\{ v_i : i < m \} \subseteq V_n$. So $V$ is a 
weighted set of elements of $G_n(\epsilon)$. Considered as a weighted set of elements of $G_\infty(\epsilon)$, 
$V$ is good, and indeed excellent. We need to check that both remain true in $G_n(\epsilon)$. 
Goodness is easy: for any $x \in G_\infty(\epsilon)$, $\sum \{ \gamma_i : \trv(x,V) = 1 \}$ is 
either $<\epsilon$ or $\geq 1-\epsilon$, and so a fortiori this is true for the $x \in G_n(\epsilon)$. 
In other words:

\begin{quotation} ($\star$) if $A$ is a weighted set of elements of $G_n(\epsilon)$, and  
$A$ is a weighted $\epsilon$-good set in $G_\infty(\epsilon)$, 
then also $A$ is a weighted $\epsilon$-good set in $G_n(\epsilon)$. 
\end{quotation}  

The subtlety about excellence is that we need to quantify over all $\epsilon$-good weighted sets of the given graph, 
and a priori those may be different in $G_n(\epsilon)$. \noteom{We show they are in fact the same.} Suppose $B$ were  
an $\epsilon$-good weighted set in $G_n(\epsilon)$ which ceased to be good in $G_{n+1}(\epsilon)$. 
(Since we made no assumptions about the minimality of $n$ this is the general case.) This would mean 
that we added some new vertex which split $B$ in an balanced way. But the vertices we add 
are all $\epsilon$-vertices, that is, they arise as weighted $\epsilon$-excellent sets of $G_n(\epsilon)$. 
In particular they have $\epsilon$-majority opinions on any weighted $\epsilon$-good set. We have shown that:  

\begin{quotation}
\noindent $(\star\star)$ If $C$ is a weighted set of $G_n(\epsilon)$, then 
$C$ is an $\epsilon$-good weighted set of $G_\infty(\epsilon)$ if and only if it is an $\epsilon$-good weighted set of $G_n(\epsilon)$.  
\end{quotation}
So $V$ is a weighted $\epsilon$-excellent set of $G_n(\epsilon)$, i.e. an $\epsilon$-vertex of $G_n(\epsilon)$.  
\end{proof}

	\begin{cor} \label{c:claim1x}
	If $v$ is an $\epsilon$-vertex of $\mcg_\infty(\epsilon)$ then $v \in V_\infty$.
	\end{cor}

	\begin{cor} \label{c:claim2x} If
	$\trv$ is an $\epsilon$-type of $\mcg_\infty(\epsilon)$ then $\trv$ is finitely realized 
	in $\mcg_\infty(\epsilon)$. 
	\end{cor}

	The subtlety of the next claim is that simply applying Littlestone Minimax to an $\epsilon$-type 
	returns some weighted $\epsilon$-good $A$, but Claim \ref{c:claim2x} only tells us that weighted $\epsilon$-\emph{excellent} 
sets belong to $B$. 

	\begin{claim} \label{c:claim4x}
	If $\trv$ is finitely realized in $\mcg_\infty(\epsilon)$ then $\trv$ is realized in $\mcg_\infty(\epsilon)$. 
	\end{claim}

	\begin{proof}
	Recall that $V_\infty$ denotes the vertex set of $\mcg_\infty(\epsilon)$. 
	Let $\trv$ be an $\epsilon$-type of $\mcg_\infty(\epsilon)$ and by \ref{extending-types},  
	let $\trv_*$ be an extension to the domain 
	\[ \mbox{ $V_\infty \cup \{ B : B $ is a weighted $\epsilon$-good set of $\mcg_\infty(\epsilon)$ $\}$. } \] 
[Recall that weighted $\epsilon$-good sets are always finite.] 
	Since by construction $\trv_*$ is finitely realized by $\epsilon$-vertices of 
	$\mcg_\infty(\epsilon)$, by \ref{c:claim1x}, $\trv_*$ is finitely realized in $\mcg_\infty(\epsilon)$. 
	Hence by Littlestone Minimax $\trv_*$ is realized by some weighted $\epsilon$-good $A$. But 
	now by construction $A$ has an opinion about every weighted $\epsilon$-good $B$ of $\mcg_\infty(\epsilon)$. 
	Hence $A$ is a weighted $\epsilon$-excellent set, in other words an $\epsilon$-vertex, which realizes the type. 
	Invoke \ref{c:claim1x} one more time to conclude $A \in V_\infty$, which completes the proof.  
	\end{proof}

	We can now complete the proof of the theorem:

	\begin{proof}[Proof of \ref{t:saturation2}] By Claims \ref{c:claim2x} and \ref{c:claim4x}.  
	\end{proof}

	\sbr

	 Theorem $\ref{sat:hypothesis}$ goes through for graphs with a similar proof: we need 
to make the natural updates to use ``excellent'' rather than ``good'' and to deal with the fact that 
things are no longer bipartite.

\begin{lemma} Let $s \in \mathbb{N}$, $s \geq 2$, and $\epsilon < 1/2^s$. 
Let $\langle a_\eta : \eta \in {^{s>}2} \rangle$ and $\langle b_\rho : \rho \in {^s 2} \rangle$ 
be a special tree of height $s$ in $G_\infty(\epsilon)$. Then there is also a special tree of height 
$s$ in $G$.
\end{lemma}

\begin{proof}
Let $m <\omega$ be minimal so that 
$\{ a_\eta : \eta \in {^{s>}2} \}\cup \{ b_\rho : \rho \in {^s 2} \} \subseteq G_m(\epsilon)$; 
if $m = 0$ we are done, so assume $m = n+1$ for some $n \in \mathbb{N}$. For each 
$\eta \in {^{s>}2}$ choose a corresponding $\epsilon$-vertex $A_\eta$ of $G_n$ and for each 
$\rho \in {^{\rho}2}$ choose a corresponding $\epsilon$-vertex $B_\rho$ of $G_n$, 
so for $\eta \tlf \rho$, $\trv(A_\eta, B_\rho) = 1$ if and only if there is an edge between 
$a_\eta$ and $b_\rho$. 

Since these are virtual vertices and virtual edges, there are two steps. 
First, for each $\rho \in {^s 2}$, define an error set 
$E_\rho = \{ (b_i, \gamma_i) \in B_\rho : \trv(b_i, A_\eta) \neq \trv(B_\rho, A_\eta) \}$ 
which will have weight $<\epsilon s$. So there is $(b, \gamma) \in B_\rho \setminus E_\rho$ 
and we can set $b^\prime_\rho := b$. 
Second, now we are in the setting of having a special tree in $G_n$ with internal nodes 
$\langle A_\eta : \eta \in {^{s>}2} \rangle$ which may be virutal, 
and leaves $\langle b^\prime_\rho : \rho \in {^s 2} \rangle$ which are vertices of $G_n$, 
so as in \ref{sat:hypothesis}(2) we find a special tree of height $s$,   
$\langle a^\prime_\eta : \eta \in {^{s>}2} \rangle$, 
$\langle b^\prime_\rho : \rho \in {^{s}2} \rangle$ in the vertices of $G_n$, 
contradicting the choice of $n$.  
\end{proof}

	\begin{concl}
	    Let $\mcg$ be a graph. 
\begin{enumerate}
\item If $\ldim(\mcg) := \ell < \infty$, and 
$\epsilon < 1/2^{\ell + 1}$, then 
\[ \ldim(\mcg) = \ldim(\mcg_\infty(\epsilon)) = \ell. \]
\item If in addition $\epsilon < 1/2^{\ell+2}$, then 
\[ \thr(\mcg) = \thr(\mcg_\infty(\epsilon)). \]
\item If $\vc(\mcg) := d < \infty$, and $\epsilon < 1/2^{d+1}$, then 
\[ \vc (\mcg ) = \vc(\mcg_\infty(\epsilon)). \]

	    \end{enumerate}
	    \end{concl}

	\begin{proof}
	 To prove this, we can use a parallel series of arguments as in \ref{sat:hypothesis}, to transfer a full special tree, half graph, or VC-tree, from $\mcg_{n+1}$ to $\mcg_{n}$, and inductively, all the way back to $\mcg$.
	\end{proof}

We have seen that $\epsilon$-saturation for graphs behaves very well for small $\epsilon$.
What about intermediate $\epsilon$? Can it still happen that bounding $\epsilon$ in terms of 
the $VC$ dimension gives us some control over the Littlestone dimension of $\mcg_\infty(\epsilon)$? 

\begin{rmk} \label{int-epsilon-graphs} 
In considering extending the arguments of \S \ref{s:intermediate-epsilon} to graphs, proving an analogue of 
Claim \ref{claim:k-realizability-backtrack} doesn't make sense globally since the domain of a function 
$($that is, the neighborhood of a vertex$)$ in $G_\infty$ includes many vertices which aren't in $G$, 
hence $k$-realizability doesn't make sense out of the box. We leave it as an interesting open question 
whether it is possible to conclude that $\mcg_\infty(\epsilon)$ remains Littlestone based only on 
restricting $\epsilon$ as a function of Littlestone dimension $d$.     
\end{rmk}

	\section{Examples} \label{s:examples1}

	In this section and the next we work out some simple and some less simple 
	examples relevant to the paper.  
Any singleton is a good set, in any graph (and for any value of $\epsilon$). 

	\subsection{Cliques} \label{sa:cliques}
	There is a slight asymmetry between cliques and anticliques since we assume that graphs
	have no loops.
	If $G$ is an anticlique, every finite subset is $\epsilon$-good.
	But also notice that every good subset of $G$ realizes a type which is already realized
	by a vertex in $G$. So $G_\infty(\epsilon) = G$.

	If $G$ is a clique, every singleton and 
	every sufficiently large finite set $X \subseteq G$, i.e. of size at least $\eps^{-1}$,
	is $\eps$-good, indeed $\eps$-excellent. There exist excellent sets which realize types not already realized in $G$,
	namely that of connecting to every $v \in G$ (since no vertex of $G$ is connected to itself).
	Thus, if $\epsilon < 1/2$ and $|G| > 1/\epsilon$,
	we will have that $G_1(\epsilon)$ is $G$ along with one new vertex connected to all elements of $G$ and to itself, and
	$G_n(\epsilon) = G_{n+1}(\epsilon)$ for all $n \geq 1$.

	\subsection{Matchings} \label{sa:matchings} 
	Suppose $\mch$ is a hypothesis class representing an infinite matching,
	i.e. $\mch = \{ h_x : x \in X \}$ and $h_x(z) = 1$ iff $z=x$.
	If $H\subseteq\mch$ is such that $|H| > \epsilon^{-1}$ then $H$ is $\epsilon$-good. Thus, any singleton and any sufficiently large finite subset of $\mch$ is good. Note that, for sufficiently small $\eps$, all good sets which are not singletons represent the constant-zero function.

	Similarly to \S \ref{sa:cliques}, 
	in an infinite matching,
	for any $\epsilon$, $\mch_1$ is $\mch$
	adjoin one new hypothesis, the constant-zero function; and
	$\mch_n = \mch_{n+1}$ for all $n \geq 1$.

	A similar analysis holds for finite matchings, where say
	$|X| = |\mch| \in \mathbb{N}$.
	If $|\mch| \leq 1/\epsilon$, there are no new good sets, and
	if $|\mch| > 1/\epsilon$, the constant-zero function appears. Note that in this second case, the constant-zero function is what we have called an $\epsilon$-type, but it is \emph{not} finitely realizable in $\mch$
	since there is no function $h \in \mch$ which agrees with it on the finite set $X$.  This example continues in \S \ref{example12}.

	\subsection{Random graphs} \label{sa:random-graphs}
	We say ``$G$ is an infinite random graph'' to mean that $G$ satisfies the
	model-theoretic extension axioms $\trg$.\footnote{The axioms $\trg$ say: for each $n$, there is an axiom saying there are at least $n$ elements, and
	for each $n,m$ there is an axioms saying that for any two disjoint sets of
	$n$ and $m$ elements respectively, there is an element connected to all
	vertices in the first set and none in the second. There is a unique countable model
	up to isomorphism, which is isomorphic with probability 1 to the graph obtained by
	starting with countably many vertices and a fair coin and flipping the coin for each pair
	to determine whether or not this is an edge. This explains the model-theoretic name ``random graph'';
	however, with the axiomatic approach, one has models of arbitrarily large infinite sizes.}

	Suppose $G$ is an infinite random graph. Then:

	\begin{enumerate}
	\item any singleton is a good set.

	\item any finite $X \subseteq G$ of even size is not a good set.

	\item any finite $X \subseteq G$ of odd size $2k+1$, $k \geq 1$ is not $\frac{k}{2k+1}$-good.

	\end{enumerate}

	Thus, if $\epsilon \leq \frac{1}{3}$, the only finite sets which are $\epsilon$-good are the singletons.
	This follows by the extension axioms, and the fact that $\frac{k}{2k+1}$ attains its minimum value
	at $k=1$.
	Hence for $\epsilon \leq \frac{1}{3}$, $G = G_\infty(\epsilon)$.

	\subsection{Half-graphs} \label{sa:half-graphs}
	Suppose $G$ is an infinite half-graph, that is, $V$ is the disjoint union
	$A =\{ a_i : i \in \mathbb{N}\} \cup B = \{ b_j : j \in \mathbb{N} \}$, and
	$a_i \sim b_j$ iff $i<j$.  We make no assumptions about
	other edges.\footnote{By convention, ``half-graph'' we will mean that
	the bi-induced pattern is a half graph, and will make no assumptions about the edges among the
	$a$'s or among the $b$'s.}  Then:

	\begin{enumerate}

	\item any singleton is a good set.

	\item if $X \subseteq G$ is finite but not a singleton, and $X \subseteq A$ or
	$X \subseteq B$, then $X$ is not $\frac{1}{3}$-good.

	\item if $X \subseteq G$ is finite, and $X \cap A \neq \emptyset$ and $X \cap B \neq \emptyset$,
	then $X$ is not $\frac{1}{6}$-good.
	\end{enumerate}

	Thus, if $\epsilon \leq \frac{1}{6}$, the only finite sets which are $\epsilon$-good are the
	singletons. Here also $G(\epsilon)=G_{\infty}(\epsilon)$.

	\begin{proof}
	Since Littlestone is central to the paper and is characterized via half-graphs, we work this out precisely.
	For (1), we use the fact that all singletons are $\eps$-good. \\

	\noindent For (2), if $X\subseteq A$, let $i_{m}$ be the median of $\{i:a_{i}\in X\cap A\}$. Then, $b_{i_{m}}$ partitions $X$ in a $\frac{1}{3}$-balanced way, since the smallest part has size
	\[\left\lfloor \frac{|X|}{2} \right\rfloor=\begin{cases}
	    \frac{1}{2}\cdot |X| & \text{$|X|$ is even}\\
	    \left(\frac{1}{2}-\frac{1}{2|X|}\right)\cdot |X|\geq \frac{1}{3}\cdot|X| & \text{$|X|$ is odd}
	\end{cases}\]

	\noindent Therefore, $X$ is not $\frac{1}{3}$-good. The case where $X\subseteq B$ is analogous.\\

	\noindent For (3), one of $A\cap X$ and $B\cap X$ has size at least $\frac{1}{2}|X|$. Without loss of generality, let it be $A\cap X$. As in (2), there some $b\in B$ splitting $A\cap X$ in an $\frac{1}{3}$-balanced way. Then, $b$ partitions $X$ in a $\frac{1}{6}$-balanced way, since the minimum of
	\[\left\{|N(b,B\cap X)|+|N(b,A\cap X)|,\hspace{10pt} |N(b,B\cap X)^{c}|+|N(b,A\cap X)^{c}|\right\}\]
	has size at least $\frac{|A\cap X|}{3}\geq \frac{|X|}{6}$.
	\end{proof}

	\sbr

	In the next examples, it will be convenient to display types as $(A,B)$ where $A \cup B = X$, $A$ are the elements labeled $1$ and $B$ are the elements labeled $0$.

	\subsection{An example where $\mch_0 \neq \mch_1 \neq \mch_2$.} \label{example15}

	We work out a basic example to show that saturation can take more than one step. This will be generalized in the next example, but may be simpler to read in this first case.

	Let $\eps\in(\frac{1}{3},\frac{1}{2})$, and consider the following hypothesis class $\mathcal{H}$:
	\begin{itemize}
	    \item $\mathcal{X}=\{x_{1},x_{2},x_{3},x_{4}\}$
	    \item $\mathcal{H}=\{h_{(i,j)}:1\leq i<j\leq 4\}$
	    \item $h_{(i,j)}(a_{\ell})=1$ iff $\ell\in\{i,j\}$.
	\end{itemize}
	In words, the three hypotheses with $i$ in the subscript connect to $x_{i}$, and to exactly one of the other $x_{j}$'s. From $\mathcal{H}_{0}$ to $\mathcal{H}_{1}$ to $\mathcal{H}_{2}$, the types realized over $\mathcal{X}$ progress from 2-element subsets, to 1-element subsets, to the empty subset.
	\begin{itemize}
	    \item Analysis of $\mathcal{H}_{0}=\mathcal{H}$: types realized by $\mathcal{H}$ over $\mathcal{X}$ are $\{(A,B):|A|=2\}$.
	    \item Analysis of $\mathcal{H}_{1}$: Additional to the previously realized types,
	    \begin{itemize}
		\item The elements $H_{\ell}=\{h_{(i,j)}: \ell\in\{i,j\}\}$, for $1\leq\ell\leq 4$ were added, and each realizes some type $(A,B)$ with $|A|=1$. If $\ell=k$, then $e(x_{\ell},H_{\ell})=|H_{\ell}|$ and $H_{\ell}(x_{k})=1$. Otherwise,  $e(x_{\ell},H_{\ell})=\frac{1}{3}|H_{\ell}|<\eps|H_{\ell}|$, and $H_{\ell}(a_{k})=0$. Therefore, $H_{\ell}$ realizes $(\{x_{\ell}\},\{x_{i}:i\neq \ell\})$.
		\item No element of $\mathcal{H}_{1}$ realizes $(\emptyset,X)$. Indeed, if some $\eps$-good set $S\subseteq\mathcal{H}_{0}$ of size $m$ realized this type, then $e(S,\mathcal{X})=2m$. By Pigeonhole principle, there is $x\in\mathcal{X}$ such that $e(x,S)\geq \frac{2m}{4}\geq \frac{1}{2}$, and $\mathbf{t}(x,S)=1$ must hold.
	    \end{itemize}
	    \item Analysis of $\mathcal{H}_{2}$: $h=\{H_{1},H_{2},H_{3},H_{4}\}$ was added. Additional to the previously realized types, it realizes $(\emptyset, X)$ because $e(x_{i},h)=\frac{1}{4}|h|<\eps|h|$ for all $1\leq i\leq 4$.
	\end{itemize}

	To summarize this example: we started with a set $X$ of size $4$, our hypotheses corresponded to all $2$-element subsets of $X$; so notice that any $Y \subseteq X$ of size two is shattered. At stage zero, every $h \in \mch$ has degree two. At stage one, we acquire some hypotheses of degree one, and at stage two, we acquire a hypothesis of degree zero. This descent of degree suggests why the $\vc$ and Littlestone dimension don't increment even though we are adding new sets. This becomes quite a bit more general in the next example.

	\subsection{An example where $k$ extension steps are needed before saturation}

	Analogous to the previous example, we can construct finite classes $(X,\mch)$ that need an arbitrary finite number of steps to occur until saturation.
	\begin{lemma}
	    Let $k\in\mathbb{N}$, and $\eps\in(0, \frac{2}{2k-1})_{\mathbb{R}}$. Then, there exists $n\in\mathbb{N}$ and $(X,\mathcal{H}_{0})$ such that $|X|=n$ and $\mathcal{H}_{0}\subsetneq\mathcal{H}_{1}\subsetneq\ldots\subsetneq\mathcal{H}_{k}$.
	\end{lemma}
	\begin{proof}
	    First, we can find some $n>2k-2$ such that $\eps\in(\frac{1}{n-k+1},\frac{2}{n})$, because
	    \footnotesize{$$\bigcup \left\{\left(\frac{1}{n-k+1},\frac{2}{n}\right)=\left(0,\frac{2}{2k-1}\right): n>2k-2\right\}.$$}
	    \normalsize
	    Having chosen $n$, we can let $X=\{a_{i}:i\in[n]\}$ and $\mathcal{H}_{0}=\{b^{k}_{S}: \;S\in\binom{[n]}{k}\}$, where
	    $b_{S}^{k}(a_{i})= [[i\in S]]$. The types realized by $\mathcal{H}_{0}$ over $X$ are $\{(A,B):|A|=k\}$. By the choice of $\epsilon$, we have $\mathcal{H}_{1}\supseteq \mathcal{H}_{0}\cup\{b_{S}^{k-1}:\;S\in\binom{[n]}{k-1}\} \text{ , where }b_{S}^{k-1}:=\{b_{S'}^{k}:\;S\subset S'\}$.\\
	    The elements $b_{S}^{k-1}$ are $\eps$-good sets of size $n-k+1$, realizing over $X$ the types $\{(A,B):|A|=k-1\}$, because $\eps>\frac{1}{n-k+1}$ allows for
	    \footnotesize{
	    \[e(a_{i},b_{S}^{k-1})=\begin{cases}n-k+1 & i\in S\\ 1 & i\not\in S\end{cases}\hspace{0.5cm}\Rightarrow\hspace{0.5cm} \mathbf{t}(a_{i},b_{S}^{k-1})=\begin{cases}1 & i\in S\\ 0 & i\not\in S\end{cases}\]}
	    \normalsize
	    Hence, no $b_{S}^{k-1}$ is in $\mathcal{H}_{0}$, implying $\mathcal{H}_{0}\subsetneq\mathcal{H}_{1}$. The proof of $\mathcal{H}_{\ell-1}\subsetneq\mathcal{H}_{\ell}$ is similar for each $2\leq\ell\leq k$, and reveals which types are realized at each step:
	    \begin{itemize}
		\item We can define $b_{S}^{k-\ell}:=\{b_{S'}^{k-\ell+1}: S'\in\binom{n}{k-\ell+1},\;\; S\subset S'\}$ for each $S\in\binom{n}{k-\ell}$, which is an $\eps$-good set of $(n-k+\ell)$ elements in $\mathcal{H}_{\ell-1}$, by $\eps>\frac{1}{n-k+\ell}$ and
		\footnotesize{
		\[e(a_{i},b_{S}^{k-\ell})=\begin{cases}n-k+\ell & i\in S\\ 1 & i\not\in S\end{cases}\hspace{0.5cm}\Rightarrow\hspace{0.5cm} \mathbf{t}(a_{i},b_{S}^{k-\ell})=\begin{cases}1 & i\in S\\ 0 & i\not\in S\end{cases}\]}
		\normalsize
		\item Hence $\mathcal{H}_{\ell}\supseteq\mathcal{H}_{\ell-1}\cup\{b_{S}^{k-\ell}:S\in\binom{n}{k-\ell}\}\supsetneq\mathcal{H}_{\ell-1}$, by the fact that elements in $\mch_{\ell}\setminus\mch_{\ell-1}$ realize $\{(A,B):|A|=k-\ell\}$, and by the claim below.
		\begin{claim}
		For every $1\leq\ell\leq k$, the elements in $\mathcal{H}_{\ell-1}$ only realize types
		$\{(A,B):|A|\geq k-\ell+1\}$ over $X$.
		\end{claim}
		\begin{proof}
		    We use induction over $\ell$. For $\ell=1$, $\mathcal{H}_{0}$ indeed realizes types in $\{(A,B):|A|\geq k\}$, since $\mathcal{H}_{0}=\{b_{S}^{k}:S\in\binom{[n]}{k}\}$ and $e(b_{S}^{k},X)=k$ for all such $S$. For the inductive step, any element of $\mathcal{H}_{\ell-1}$ is of one of three kinds. For each, we argue the type realized is of size $\geq k-\ell+1$.
		    \begin{enumerate}
			\item Elements $b_{S}^{k}, \ldots, b_{S}^{k-\ell+1}$ have $e(b_{S}^{\cdot},X)\geq k,\ldots,  k-\ell+1$, respectively.
			\item Element $b_{S}^{p}\in\mathcal{H}_{p}$ added at step $p\leq \ell-2$, satisfies $e(b_{S}^{p},X)\geq k-\ell+2$, by inductive hypothesis.
			\item An $\eps$-good set $S\subseteq\mathcal{H}_{\ell-2}$ added as $b\in\mathcal{H}_{\ell-1}$. If $|S|=m$, then $e(X,S)\geq (k-\ell+2)m$, since types realized by $\mathcal{H}_{\ell-2}$ are of size $\geq k-\ell+2$. There must exist  $a_{i_{r}}\in X$ for $r\in [k-\ell+1]$, such that $e(a_{i_{r}},S)\geq \frac{2m}{n-k+\ell}>\frac{2}{n}$ for each $r$. Else, from counting in two ways, we get a contradiction
			\small{
			$$(k-\ell+2)m \leq e(X,S)
			    < m(k-\ell) + \frac{2m}{n-k+\ell}\cdot(n-k+\ell).$$}
			\normalsize
			Then, $\eps<2/n$ implies $\mathbf{t}(a_{i_{r}},b)=\mathbf{t}(a_{i_{r}},S)=1$ for all $r\in[k-\ell+1]$.
		    \end{enumerate}
		\end{proof}
	    \end{itemize}
	    This completes our proof that $\mathcal{H}_{0}\subsetneq\ldots\subsetneq\mathcal{H}_{k}$ holds for the constructed $(X,\mathcal{H}_{0})$.
	\end{proof}

	{These examples give a family of Littlestone classes where the saturation construction does not stabilize until level $k$. Note that this is arranged at the cost of higher and higher initial, though obviously still finite, $\vc$-dimension.}

	{Also note that $\vc(\mch_{i})=\ldim(\mch_i)=k$ for all $i<\omega$, and in particular $\vc(\mch_{\infty}(\eps))=\ldim(\mch_{\infty}(\eps))=k$. On one hand, $\vc(\mch_{i})\leq\ldim(\mch_i)\leq k$ can be shown by noticing that all virtual elements realize types of size at most $k$. On the other hand, $\ldim(\mch_{i})\geq\vc(\mch_{i})\geq \vc(\mch_{0})\geq k$, because any} subset
	
	{of $X$ of size $k$ can be shattered by elements $b_{S}^{k}\in\mch_{0}$ as long as $|X|\geq 2k$ (a property of the above constructions, as long as $\eps<1/k$).}
	\sbr

\section{Examples for large $\epsilon$} \label{s:examples2}

In this section we prove by example that for ``large'' epsilon, the saturation of a Littlestone class can even have infinite VC dimension. 

	\subsection{An example where VC-dimension goes to infinity}
	The following claim shows that if $\eps>0$ is large enough, it is possible for VC-dimension to increment at every step of saturation:
	\begin{example}\label{eg:counterex-vcblowsup}
	    Consider the class $(X,\mch)$ where $X:=\mathbb{N}$, $\mch_{0}:=\{b_{S}\subseteq\mathbb{N}:|S|=d\}$, and $b_{S}(i):=[[i\in S]]$ for any $S\subset\mathbb{N}$. Then, $\vc(\mch_{0})=\ldim(\mch_{0})=d$. What's more, if $\eps>\frac{1}{d+1}$, then $\ldim(\mch_{n})\geq \vc(\mch_{n})\geq n+d$ for all $n<\omega$. In particular, $\vc(\mch_{\infty}(\eps))=\omega$.
	\end{example}
	\begin{proof}
	    To show $\vc(\mch_{0})=\ldim(\mch_{0})=d$, all elements of $\mch$ have degree $d$ and infinite codegree, so any set of $d$ elements in $X$ can be shattered. We show by induction on $n<\omega$ that $\mch_{n}\supseteq\{b_{S}\subseteq\mathbb{N}:d\leq |S|\leq n+d\}$, which would imply $\ldim(\mch_{n})\geq\vc(\mch_{n})\geq n+d$. The case $n=0$ holds by construction. Suppose $\mch_{n}\supseteq\{b_{S}\subseteq\mathbb{N}:d\leq|S|\leq n+d\}$ for some $n<\omega$. Then, for all $A\subseteq\mathbb{N}$ with $|A|=n+d+1$, the set
	    \[B_{A}:=\{b_{S}:|S|=n+d,S\subseteq A\}\subseteq\mch_{n}\]
	    is of size $(n+d+1)$, and for all $i\in\mathbb{N}$, either $e(i,B_{A})=0$ when $i\not\in A$, or $e(i,B_{A})=n+d$ when $i\in A$. Since $\eps>\frac{1}{d+1}\geq\frac{1}{n+d}$, then $\trv(i,B_{A})=[[i\in A]]$, is defined, $B_{A}$ is $\eps$-good, and $b_{A}\in\mch_{n+1}$ is added.
	    Thus, $\mch_{n+1}\supseteq\{b_{S}:d\leq|S|\leq (n+d+1)\}$. To show that $\ldim(\mch_{n})\geq \vc(\mch_{n})\geq n+d$, we use the same reasoning as for $\mch_{0}$.
	\end{proof}

	To summarize a key idea of this example: Suppose $X$ is infinite and $\mch$ is the set of (characteristic functions of) $k$-element subsets of $X$. Consider $\{ a_1, \dots, a_{n+1} \} \subseteq X$.
	We would like to realize some pattern of zeros and ones over it, call this pattern $\sigma : \{ a_1, \dots, a_{n+1}\} \rightarrow \{ 0, 1\}$. By assumption on $\mch$, for every $1 \leq i \leq n+1$, there is a function $h_i$ which agrees with $\sigma$ on
	$\{ x_1, \dots, x_{n+1} \} \setminus \{ x_i \}$.
	Consider $H = \{ h_i : 1 \leq i \leq n+1 \}$. (We can choose the $h_i$'s to be distinct, or we can make this a weighted set.) For every $i$, at least $n/(n+1)$ of the functions vote correctly. So if $\eps > 1/n$, the majority vote of $H$ is well defined and agrees with $\sigma$.  In this way we progressively extend the shattering to larger sets.

	\sbr
	\subsection{An example of a good set which is not an excellent set}  \label{example:good-excellent}
	Excellent sets play an important role in the sections on graphs. For more discussion and examples on this see \cite{Malliaris-Moran}.
	Let $\eps=\frac{21}{60}$, and $G=(V,E)$ be the graph shown below, with $V=A\cup B$, $A=\{a_{i}\}_{i\in[6]}$, and $B=\{b_{j}\}_{j\in[6]}$.

	\begin{figure}[ht]
		\centering
		\includegraphics[width=0.4\linewidth]{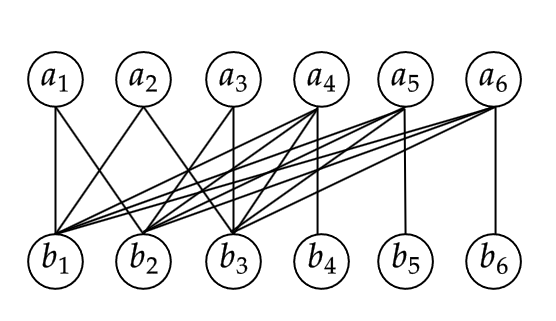}
		\label{fig:good-not-exc}
	\end{figure}
	While $A$ and $B$ are both $\varepsilon$-good, neither is $\varepsilon$-excellent. Indeed, all vertices have 0, 1, or 5 neighbors in $A$, and 0, 2 or 4 neighbors in $B$, and $\{\frac{i}{6}: i\in\{0,\ldots, 6\}\setminus\{3\}\}\subseteq [0,\varepsilon)\cup(1-\varepsilon,1]$.  Moreover, $\mathbf{t}(A,B)$ is not defined because $B$ is $\varepsilon$-good and $|\{a\in A: \mathbf{t}(a,B)=0\}|=|\{a\in A: \mathbf{t}(a,B)=1\}|= \frac{1}{2}\cdot|A|$. Similarly, $\mathbf{t}(B,A)$ is not defined because $A$ is $\varepsilon$-good and $|\{b\in B: \mathbf{t}(b,A)=0\}| =|\{b\in B: \mathbf{t}(b,A)=1\}|= \frac{1}{2}\cdot|B|$.

	\subsection{Extensions} Observe that
	in both cliques $\ref{sa:cliques}$ and matchings $\ref{sa:matchings}$,
	given any two good sets, there is a larger good set containing both of them. This is a property that all bounded degree graphs enjoy:

	\begin{propn}
	    Let $G=(V,E)$ be an arbitrary infinite graph such that every vertex has degree at most $d\in\mathbb{N}$. Then, every finite $X\subseteq V$ with $|X|> \frac{d}{\eps}$ is $\eps$-good. Moreover, for any two $\eps$-good sets $X_{1}$ and $X_{2}$, there exists an $\eps$-good set $X_{3}$ such that $X_{1}\cup X_{2}\subseteq X_{3}$.
	\end{propn}
	\begin{proof}
	    For the first part, consider such vertex set $X$, and any $v\in V$. Since $\deg(v)\leq d$, $v$ partitions $X$ in an $\eps$-unbalanced way, since the smaller size of this partition is less than or equal than $d<\eps\cdot|X|$. \\

	    \noindent For the second part, consider two $\eps$-good sets $X_{1}$ and $X_{2}$. Infiniteness of $|V|$ implies there is a set $X_{3}\subset V$ containing $X_{1}\cup X_{2}$, possibly with some vertices added in order to have $|X_{3}|>\frac{d}{\eps}$ hold. By the first claim, $X_{3}$ is $\eps$-good.
	\end{proof}

	\sbr
	\subsection{Realizability and the value of $\epsilon$} \label{example12}
	Suppose $\mch$ is a hypothesis class and $r < 1/\epsilon$.
	Observe that any function which is $r$-realized in
	$\mch_1(\epsilon)$ is already realized in $\mch$. [Why?
	By a union bound, see \ref{claim:k-realizability-backtrack} for a
	more general result.]  On the other hand,
	consider the class of singletons over $X = \{ 1, 2, 3, 4, 5\}$,
	$\epsilon = 1/4$, and $r = 5 > 1/\epsilon$. In this example,
	$H = \mch$ (with the counting measure)
	is an $\epsilon$-good set, and $\trv_H$ is the constant
	zero function, which is realized in $\mch_1(\epsilon)$ and even
	$4$-realized in $\mch$, but not realized, indeed not even $5$-realized, in $\mch$.

	\sbr
	\section{Open questions}

	There are some very interesting combinatorial questions arising from this analysis around the regimes of $\epsilon$, 
which don't have obvious analogues in infinitary notinos of saturation. 
	Theorem \ref{thm:littlestone-Cd} draws our attention to the case 
	$1/Cd \geq \epsilon \geq 1/(d+1)$. The first few questions are about different aspects of this gap. 
	In each case we assume $\ldim(\mch) = \ell < \infty$ and hence $\vc(\mch) = d < \infty$, and in 
	each case $\mch_\infty$ means $\mch_\infty(\epsilon)$. The constant $C$ is from \ref{thm:littlestone-Cd}.

	\begin{qst}
    If $\ldim(\mch)<\infty$, can it happen that $\ldim(\mch_{\infty})=\infty$ but $\mch_{\infty}$ still has the property that $\vc(\mch_{\infty})<\infty$?
	\end{qst}

	\begin{qst}
    In general, what happens to $\ldim(\mch_\infty(\eps))$ when $ 1/Cd \leq \eps$?
	\end{qst}

It follows from the definitions that if ever $\mch_n = \mch_{n+1}$ then $\mch_{n+1} =\mch_{m}$ for all $m > n$.  
However, our only example above where $\mch_n \neq \mch_{n+1}$ for all $n<\omega$ is in the case where 
$\vc(\mch_\infty) = \infty$. The next question asks if this can happen in a tamer setting. 
(By the above, sufficiently small values of $\epsilon$ can be ruled out.)

We can also ask the parallel question for graphs $G$.

\begin{qst}
What happens to the Littlestone dimension of $\mcg_\epsilon(\infty)$ when  
$\mcg$ is a Littlestone class and $\epsilon$ is of intermediate range bounded as a function of the 
$\vc$ dimension $d$, per Remark $\ref{int-epsilon-graphs}$?
\end{qst}

\begin{qst}
    Are there any other preservation theorems for the saturation with respect to other combinatorial complexity measures? 
\end{qst}

\br
\bibliographystyle{amsplain}

\end{document}